\documentclass[11pt,a4paper]{article}
\usepackage[T1]{fontenc}
\usepackage[utf8]{inputenc}
\usepackage[english]{babel}
\usepackage{lmodern,microtype}
\usepackage{geometry}
\usepackage{amsmath,amssymb,amsthm,mathtools}
\usepackage{enumitem}
\usepackage{xcolor}
\usepackage{hyperref}
\usepackage[nameinlink,noabbrev]{cleveref}
\hypersetup{colorlinks=true,linkcolor=blue!55!black,citecolor=blue!55!black,urlcolor=blue!55!black}

\newtheorem{theorem}{Theorem}[section]
\newtheorem{proposition}[theorem]{Proposition}
\newtheorem{lemma}[theorem]{Lemma}
\newtheorem{corollary}[theorem]{Corollary}
\theoremstyle{definition}
\newtheorem{definition}[theorem]{Definition}

\theoremstyle{remark}
\newtheorem{remark}[theorem]{Remark}

\newcommand{\R}{\mathbb R}
\newcommand{\mug}{\mu_g}
\newcommand{\Dg}{D_g}
\newcommand{\Sg}{S_g}
\newcommand{\dg}{\delta_g}
\newcommand{\ip}[2]{\left\langle #1,#2\right\rangle}
\newcommand{\norm}[1]{\left\lVert #1\right\rVert}
\newcommand{\Pgminus}{P_g^-}
\newcommand{\Pgplus}{P_g^+}
\newcommand{\Wplus}{\mathbb W_g^+(V,V^*)}
\newcommand{\Ag}{\mathcal A_g}

\title{\textbf{Banach-Valued Stieltjes Calculus and Linear Parabolic Evolution Equations with Posterior States}}
\author{Francisco J. Fern\'andez\\[2pt]\small Departamento de Estat\'istica, An\'alise Matem\'atica e Optimizaci\'on\\\small Universidade de Santiago de Compostela, Spain}
\date{}

\begin{document}
\maketitle

\begin{abstract}
We develop a Banach-valued lateral Stieltjes calculus adapted to linear evolution equations on a Gelfand triple
\[
 V\hookrightarrow H\simeq H^*\hookrightarrow V^*.
\]
The evolution problem is formulated with the posterior state,
\[
 \Dg u(t)+A(t)\Sg u(t)=f(t),\qquad u(a)=u_0.
\]
This choice is decisive at atoms of the Stieltjes measure: if $\delta=\Delta g(r)>0$, the equation yields the implicit variational resolvent
\[
 (I+\delta A(r))u(r^+)=u(r)+\delta f(r).
\]
We prove a Banach-valued lateral product formula, a weak--strong reconstruction theorem, and a Stieltjes analogue of the Lions--Magenes energy identity. In the separable Hilbert setting considered here, the stable projection chain used in the proof is shown to exist automatically by a spectral-band construction. The latter contains the exact positive jump contribution
\[
 \frac12\sum_r\norm{u(r^+)-u(r)}_H^2.
\]
Existence is constructed independently by Galerkin approximation and by a Rothe scheme adapted to the Stieltjes mass and its atoms. The Rothe construction does not use spatial projections. We also explain why the available Aubin--Lions theorem for anterior Stieltjes--Bochner evolution graphs is not automatically applicable to the posterior energy space and, in the present linear coercive setting, is not needed. A heat equation with Stieltjes time is included as a model application. The claims are positioned relative to earlier Stieltjes parabolic equations, general measure-driven evolution, and operator-valued-measure parabolic systems; no priority is claimed for measure-driven parabolic theory or Galerkin approximation in general.
\end{abstract}


\section{Introduction}
Stieltjes differential calculus provides a convenient language for evolutions driven by a nondecreasing left-continuous clock $g$. The associated Lebesgue--Stieltjes measure may contain absolutely continuous, singular continuous, and atomic components. Consequently, the same time variable can encode ordinary continuous evolution, inactive intervals, singular clocks, and instantaneous events. For PDEs, however, the choice of the state sampled by the spatial operator at an atom is structural: an anterior convention and a posterior convention generate different jump laws.

There is substantial prior work on measure-driven evolution that must be distinguished from the specific framework developed here. Cicho\'n and Satco study measure differential inclusions for arbitrary finite Borel measures and obtain regulated or bounded-variation solutions encompassing continuous, discrete, impulsive, and Zeno-type dynamics \cite{CichonSatco2014}. Satco treats nonlinear measure-driven integral equations with regulated solutions in the Kurzweil--Stieltjes setting \cite{Satco2014}, while Cao and Sun study abstract semilinear measure-driven equations in Banach spaces by Lebesgue--Stieltjes integration and $C_0$-semigroup methods \cite{CaoSun2015}. More directly on the PDE side, Ahmed's operator-valued-measure program includes parabolic systems on Gelfand triples, weak solutions, Galerkin constructions, and regularity theory; see, in particular, Section~3 of \cite{Ahmed2007DCDIS} and the broader operator-valued-measure framework in \cite{Ahmed2007NA}. Thus the present paper does \emph{not} claim the first variational, infinite-dimensional, or Galerkin theory for measure-driven parabolic equations. The distinction is narrower: the scalar Stieltjes clock is combined with a lateral \emph{posterior-state} variational formulation, and the resulting atomic algebra is built into the energy theory and both approximation schemes.

Within the Stieltjes PDE literature closest to the present notation, Banach-valued Stieltjes--Bochner spaces and their use in parabolic equations were developed by Fern\'andez and Tojo \cite{FernandezTojo2020}. In arXiv:2001.11183v2, Definition~3.1 in Section~3, equations~(68)--(69), formulates the solution with the spatial terms evaluated at the anterior state $u(s)$. Section~4 then uses the classical diagonalization method; Theorem~4.1 gives the scalar Stieltjes equation used mode by mode and Theorem~4.2 proves existence and uniqueness under hypotheses (H1)--(H5), with the spectral reduction displayed in Part~1, equations~(78)--(81). A later application to a two-dimensional population model again uses a diagonalizable spatial differential operator and Stieltjes time to represent hatching impulses and hibernation intervals \cite{AreaFernandezNietoTojo2025}. These works establish that Stieltjes parabolic PDEs themselves are not new. The present paper instead keeps the left-continuous anterior representative but places the spatial operator at the posterior trace $\Sg u$. At an atom this produces a coercive implicit resolvent rather than an anterior multiplicative update.

The lateral structure used below is inherited from the scalar theory in \cite{FernandezLateral2026}. In arXiv:2609.10659v1, Theorem~3.3 in Section~3.3 gives the exact lateral integration-by-parts formula, Theorem~4.6 in Section~4.4 gives the scalar weak--strong reconstruction theorem, and Lemma~6.4 together with Theorem~6.5 in Section~6.2 provides the atom-compatible ordered-simplex estimate and finite-dimensional global existence theorem. We prove the Banach-valued duality formula and Banach-valued reconstruction directly, rather than treating these scalar statements as vector-valued theorems. The Stieltjes--Bochner fundamental theorem and Aubin--Lions graph compactness theory are developed in \cite{FernandezAubinLions2026}. In arXiv:2609.10064v1, the RNP-dependent fundamental theorem is Theorem~3.4, the evolution graph is Definition~4.1, the main compactness theorem is Theorem~4.7, and the Gelfand-triple specialization is Corollary~5.4. These results are used here to delimit, rather than to close, the existence proof: the anterior graph controls $u$, whereas the natural posterior energy estimate controls $\Sg u$.

A concrete modelling motivation for posterior implicit Stieltjes stepping appears in Fern\'andez and Area \cite[Section 4, Remark 4.2, equation (4.6)]{FernandezArea2026}, where the post-event algebra has the same resolvent form. That work explicitly leaves continuous well-posedness of the control-to-state map outside its scope, so none of its numerical or optimization results is imported here. More generally, energy--dissipation identities with nontrivial jump contributions are well established in other discontinuous evolution theories, for example in the Balanced-Viscosity theory of rate-independent systems \cite{MielkeRossiSavare2016}. Accordingly, our novelty claim is not that jumps can contribute to an energy balance in general. What is specific here is the exact positive quadratic correction generated by the lateral product formula for a posterior Stieltjes evolution on a Gelfand triple.

The mathematical contributions are therefore deliberately stated in a narrow form. First, we organize anterior and posterior Bochner primitives and prove the Banach-valued lateral product and reconstruction statements needed by the evolution problem. Second, we show that every separable Hilbert Gelfand triple considered here admits a stable projection chain, and use such a chain to establish a Stieltjes--Lions--Magenes identity whose atomic part is the exact term
\[
 \frac12\sum_{r}\norm{u(r^+)-u(r)}_H^2.
\]
Third, for coercive time-dependent bilinear forms we prove the posterior jump law
\[
 (I+\Delta g(r)A(r))u(r^+)=u(r)+\Delta g(r)f(r)
\]
in variational form and construct solutions independently by Galerkin approximation and by a Stieltjes-adapted Rothe scheme. The Rothe existence argument requires no stable spatial projection. Fourth, we give an explicit atomic counterexample showing that posterior $V$-control does not imply the anterior $V$-control required by the available Stieltjes--Aubin--Lions graph theorem. The results are thus best read as a posterior-state variational theory, not as a replacement for the broader measure-driven, operator-valued-measure, semigroup, or anterior Stieltjes theories cited above.

The paper is organized as follows. Sections~\ref{sec:basic}--\ref{sec:reconstruction} develop the Banach-valued lateral calculus. Sections~\ref{sec:gelfand}--\ref{sec:energy} contain the variational energy theory. Sections~\ref{sec:galerkin} and \ref{sec:rothe} give the two existence constructions. Section~\ref{sec:aubin} discusses the Aubin--Lions compatibility issue, and Section~\ref{sec:heat} treats the heat equation.

\section{Stieltjes measure, posterior trace, and Bochner primitives}\label{sec:basic}
Let $I=[a,b]$ and let $g:I\to\R$ be nondecreasing and left-continuous. Its finite Lebesgue--Stieltjes measure is normalized by
\[
 \mug([s,t))=g(t)-g(s),\qquad a\le s\le t\le b.
\]
Set
\[
 M_g:=\mug([a,b)),\qquad
 \Ag:=\{r\in[a,b):\Delta g(r)>0\},\qquad
 \dg(r):=\mug(\{r\})=\Delta g(r).
\]
The symbol $\Ag$ denotes the atomic set, whereas $\Dg u$ denotes the Stieltjes derivative of a function. The atomic set is at most countable and $\sum_{r\in \Ag}\dg(r)\le M_g$. We extend $\dg$ by zero away from the atomic set. Throughout, the underlying measure space is $([a,b),\mug)$; whenever an integral is written over $I=[a,b]$, we extend $\mug$ by zero on the endpoint $\{b\}$.

Let $X$ be a Banach space. We write $L_g^p(I;X)=L^p([a,b),\mug;X)$ for the corresponding Bochner space.

\begin{definition}[Bochner--Stieltjes absolute continuity]\label{def:ACB}
A function $U:I\to X$ belongs to $AC_g^{\rm B}(I;X)$ if there exist $x_a\in X$ and $v\in L_g^1(I;X)$ such that
\begin{equation}\label{eq:primitive}
 U(t)=x_a+\int_{[a,t)}v(s)\,d\mug(s),\qquad t\in I.
\end{equation}
The density $v$, uniquely determined $\mug$-almost everywhere, is denoted by $\Dg U$.
\end{definition}

\begin{lemma}[Uniqueness of the Bochner--Stieltjes density]\label{lem:densityunique}
If $v,w\in L_g^1(I;X)$ have the same primitive, then $v=w$ $\mug$-almost everywhere.
\end{lemma}
\begin{proof}
Let $z=v-w$. The integral of $z$ vanishes on every half-open interval. For each $x^*\in X^*$, the signed measure $E\mapsto\int_E x^*(z)\,d\mug$ therefore vanishes on the generating half-open semiring and hence on all Borel sets. Thus $x^*(z)=0$ $\mug$-almost everywhere for each fixed $x^*$. Strong measurability makes $z$ essentially separably valued; a countable norming family on the closed separable span of an essential range, extended by Hahn--Banach, produces one common null set outside which $z=0$.
\end{proof}

The half-open interval in \eqref{eq:primitive} fixes the anterior representative. At an atom $r$,
\[
 U(r^+)-U(r)=\dg(r)\Dg U(r).
\]
This suggests the posterior trace
\begin{equation}\label{eq:Sg}
 \Sg U:=U+\dg\Dg U.
\end{equation}
Thus $\Sg U(r)=U(r^+)$ at every atom and $\Sg U=U$ outside the atoms, up to $\mug$-null sets.

\begin{lemma}[Anterior and posterior primitive operators]\label{lem:Ppm}
For $1\le p\le\infty$, define
\[
 (\Pgminus w)(t):=\int_{[a,t)}w\,d\mug,
 \qquad \Pgplus w:=\Pgminus w+\dg w.
\]
For $t<b$, $(\Pgplus w)(t)=\int_{[a,t]}w\,d\mug$.
Then $\Pgminus,\Pgplus:L_g^p(I;X)\to L_g^p(I;X)$ are bounded linear operators. In particular, for $p=2$,
\[
 \norm{\Pgminus w}_{L_g^2(X)}\le M_g\norm{w}_{L_g^2(X)},\qquad
 \norm{\Pgplus w}_{L_g^2(X)}\le 2M_g\norm{w}_{L_g^2(X)},
\]
and
\[
 \norm{\Pgminus w}_{L^\infty(I;X)}\le M_g^{1/2}\norm{w}_{L_g^2(X)}.
\]
\end{lemma}

\begin{proof}
For every $t$,
\[
 \norm{(\Pgminus w)(t)}_X
 \le \int_{[a,t)}\norm{w(s)}_X\,d\mug(s)
 \le M_g^{1/2}\norm{w}_{L_g^2(X)}.
\]
Integrating its square in $t$ gives the first $L^2$ estimate. Since $0\le\dg\le M_g$,
\[
 \norm{\dg w}_{L_g^2(X)}^2
 =\int_{[a,b)} \dg(t)^2\norm{w(t)}_X^2\,d\mug(t)
 \le M_g^2\norm{w}_{L_g^2(X)}^2.
\]
The second estimate follows from the triangle inequality. More generally, for $1\le p<\infty$, H\"older's inequality gives
\[
 \|(\Pgminus w)(t)\|_X\le M_g^{1-1/p}\|w\|_{L_g^p(X)},
\]
whence $\|\Pgminus w\|_{L_g^p(X)}\le M_g\|w\|_{L_g^p(X)}$. Since $0\le\dg\le M_g$, also $\|\dg w\|_{L_g^p(X)}\le M_g\|w\|_{L_g^p(X)}$, and therefore $\|\Pgplus w\|_{L_g^p(X)}\le2M_g\|w\|_{L_g^p(X)}$. The case $p=\infty$ follows directly from the defining integral.
\end{proof}

\begin{proposition}[Basic properties of the primitive]\label{prop:primitive}
If $U\in AC_g^{\rm B}(I;X)$, then $U$ is bounded, left-continuous in $X$, and has a right limit at every $t<b$. Moreover, for each $x^*\in X^*$,
\[
 \Dg\ip{x^*}{U}=\ip{x^*}{\Dg U},\qquad
 \Sg\ip{x^*}{U}=\ip{x^*}{\Sg U}.
\]
\end{proposition}

\begin{proof}
The bound follows immediately from \eqref{eq:primitive}. If $t_n\uparrow t$, then
\[
 U(t)-U(t_n)=\int_{[t_n,t)}\Dg U\,d\mug\to0
\]
by absolute continuity of the Bochner integral. If $t_n\downarrow t$, $t_n>t$, continuity from above gives
\[
 U(t_n)-U(t)\to\int_{\{t\}}\Dg U\,d\mug=\dg(t)\Dg U(t).
\]
Thus the right limit exists and agrees with \eqref{eq:Sg} at atoms. The scalarization formulas follow because continuous linear functionals commute with Bochner integration.
\end{proof}

\begin{remark}[Radon--Nikod\'ym property]
Definition~\ref{def:ACB} does not require the Radon--Nikod\'ym property (RNP). The RNP enters only if one wishes to start from a purely metric notion of $g$-absolute continuity and derive a Bochner density. In arXiv:2609.10064v1, the corresponding metric-to-Bochner characterization under the RNP is Theorem~3.4 in Section~3, pp.~30--36 of the public arXiv PDF \cite{FernandezAubinLions2026}.
\end{remark}

\section{A Banach-valued lateral product formula}\label{sec:product}
\begin{theorem}[Lateral product rule in Banach duality]\label{thm:product}
Let $X$ be a Banach space, $U\in AC_g^{\rm B}(I;X)$, and $\Phi\in AC_g^{\rm B}(I;X^*)$. Then $q(t)=\ip{\Phi(t)}{U(t)}$ belongs to $AC_g(I)$ and
\begin{align}
 \Dg q&=\ip{\Dg\Phi}{U}+\ip{\Sg\Phi}{\Dg U},\label{eq:prod1}\\
 &=\ip{\Dg\Phi}{\Sg U}+\ip{\Phi}{\Dg U}\label{eq:prod2}
\end{align}
$\mug$-almost everywhere. Consequently, for $a\le s\le t\le b$,
\begin{equation}\label{eq:ibp}
 \ip{\Phi(t)}{U(t)}-\ip{\Phi(s)}{U(s)}
 =\int_{[s,t)}\ip{\Dg\Phi(r)}{U(r)}\,d\mug(r)
 +\int_{[s,t)}\ip{\Sg\Phi(r)}{\Dg U(r)}\,d\mug(r).
\end{equation}
\end{theorem}

\begin{proof}
Write $u=\Dg U$ and $\phi=\Dg\Phi$. On $[s,t)$,
\[
 U(r)=U(s)+\int_{[s,r)}u(q)\,d\mug(q),\qquad
 \Sg\Phi(r)=\Phi(s)+\int_{[s,r]}\phi(p)\,d\mug(p).
\]
Substitution into the right-hand side of \eqref{eq:ibp} yields two linear terms and two double integrals. The latter are absolutely integrable since
\[
 \iint\left|\ip{\phi(p)}{u(q)}\right|\,d\mug(p)d\mug(q)
 \le \norm{\phi}_{L_g^1(X^*)}\norm{u}_{L_g^1(X)}.
\]
Fubini's theorem therefore applies. The first double integral occupies the region $q<p$; the second occupies $p\le q$. These regions are disjoint and cover $[s,t)^2$. The diagonal is contained in the second region, which is essential when $\mug$ has atoms because then the product measure of the diagonal need not vanish. Hence the two double integrals add to
\[
 \ip{\int_{[s,t)}\phi\,d\mug}{\int_{[s,t)}u\,d\mug}.
\]
Combining this with the two linear terms gives
\[
 \ip{\Phi(t)-\Phi(s)}{U(s)}+\ip{\Phi(s)}{U(t)-U(s)}
 +\ip{\Phi(t)-\Phi(s)}{U(t)-U(s)},
\]
which is exactly $\ip{\Phi(t)}{U(t)}-\ip{\Phi(s)}{U(s)}$. This proves \eqref{eq:ibp}, and the scalar Stieltjes fundamental theorem gives \eqref{eq:prod1}. Finally,
\[
 \Sg U-U=\dg\Dg U,\qquad \Sg\Phi-\Phi=\dg\Dg\Phi,
\]
so the difference between the two right-hand sides in \eqref{eq:prod1} and \eqref{eq:prod2} is zero.
\end{proof}

The scalar predecessor is Theorem 3.3 of arXiv:2609.10659v1, Section 3.3 \cite{FernandezLateral2026}. The proof above is genuinely Banach-valued and makes the atomic diagonal explicit through Fubini's theorem.

\section{Weak--strong reconstruction in a Banach space}\label{sec:reconstruction}
Let
\[
 \mathcal T_g:=\left\{\varphi(t)=c+\int_{[a,t)}\psi(s)\,d\mug(s):c\in\R,\ \psi\in L_g^\infty(I)\right\}.
\]
Then $\Dg\varphi=\psi$ and $\Sg\varphi=\varphi+\dg\psi$.

\begin{theorem}[Banach-valued weak--strong reconstruction]\label{thm:reconstruction}
Let $X$ be a Banach space, $u,v\in L_g^1(I;X)$, and $u_a,u_b\in X$. The following are equivalent:
\begin{enumerate}[label=(\roman*)]
\item for every $\varphi\in\mathcal T_g$,
\begin{equation}\label{eq:weakX}
 \int_{[a,b)}u\,\Dg\varphi\,d\mug+
 \int_{[a,b)}v\,\Sg\varphi\,d\mug
 =u_b\varphi(b)-u_a\varphi(a)
\end{equation}
as an identity in $X$;
\item the function
\[
 U(t):=u_a+\int_{[a,t)}v\,d\mug
\]
satisfies $U=u$ $\mug$-almost everywhere and $U(b)=u_b$;
\item there exists a unique $U\in AC_g^{\rm B}(I;X)$ such that
\[
 U=u\quad\mug\text{-a.e.},\qquad U(a)=u_a,\quad U(b)=u_b,
 \qquad \Dg U=v\quad\mug\text{-a.e.}
\]
\end{enumerate}
No reflexivity or RNP assumption is required.
\end{theorem}

\begin{proof}
\emph{(i)$\Rightarrow$(ii).} Define $U(t)=u_a+\int_{[a,t)}v\,d\mug$. Then $U\in AC_g^{\rm B}(I;X)$, $\Dg U=v$ $\mug$-a.e., and $U(a)=u_a$. Fix $x^*\in X^*$ and $\varphi\in\mathcal T_g$, and put $\Phi=\varphi x^*$. Theorem~\ref{thm:product} gives
\[
 \int_{[a,b)}x^*(U)\Dg\varphi\,d\mug+
 \int_{[a,b)}x^*(v)\Sg\varphi\,d\mug
 =x^*(U(b))\varphi(b)-x^*(u_a)\varphi(a).
\]
Separation by $X^*$ gives the corresponding identity in $X$. Taking $\varphi\equiv1$ in (i) yields $U(b)=u_b$. Subtracting the two vector identities and using that every $\psi\in L_g^\infty(I)$ occurs as $\Dg\varphi$ for some $\varphi\in\mathcal T_g$, we get
\[
 \int_{[a,b)}(u-U)\psi\,d\mug=0\qquad\forall\psi\in L_g^\infty(I).
\]
For fixed $x^*$ choose $\psi=\operatorname{sgn}(x^*(u-U))$. Hence $x^*(u-U)=0$ $\mug$-a.e. Strong measurability gives an essentially separable range; a countable norming family and Hahn--Banach then give one common null set on whose complement $u=U$. Thus (ii) holds.

\smallskip\noindent\emph{(ii)$\Rightarrow$(iii).} This is Definition~\ref{def:ACB}: the displayed primitive belongs to $AC_g^{\rm B}$, has derivative $v$ $\mug$-a.e., the prescribed endpoints, and represents $u$ $\mug$-a.e.

\smallskip\noindent\emph{(iii)$\Rightarrow$(i).} Apply the lateral product formula to $U$ and the scalar test function; replacing $U$ by $u$ and $\Dg U$ by $v$ in $\mug$-a.e. integrals gives \eqref{eq:weakX}.

\smallskip\noindent\emph{Uniqueness.} The derivative and initial value determine the primitive, by Lemma~\ref{lem:densityunique}.
\end{proof}

\begin{remark}[Classical reduction]
For $X=\mathbb R$ and $g(t)=t$, $\mug=dt$ and $\Sg\varphi=\varphi$. The first implication reduces to the classical reconstruction from
\[
 \int_a^b u\varphi'\,dt+\int_a^b v\varphi\,dt
 =u_b\varphi(b)-u_a\varphi(a),
 \qquad U(t)=u_a+\int_a^t v(s)\,ds.
\]
\end{remark}

The scalar analogue is Theorem 4.6 of arXiv:2609.10659v1, Section 4.4 \cite{FernandezLateral2026}. The argument above records the additional separability step required to pass from scalarized vanishing to a Banach-valued identity.

\section{Variational framework and stable projections}\label{sec:gelfand}
Let
\begin{equation}\label{eq:gelfand}
 V\hookrightarrow H\simeq H^*\hookrightarrow V^*
\end{equation}
be a real Gelfand triple, where $V$ and $H$ are separable Hilbert spaces and $V\hookrightarrow H$ is continuous and dense. We identify $H$ with its canonical image in $V^*$.

Let $a:I\times V\times V\to\R$ be a family of bilinear forms satisfying, for a fixed representative defined at every $t\in I$,
\begin{enumerate}[label=(A\arabic*)]
\item $t\mapsto a(t;u,v)$ is $\mug$-measurable for all $u,v\in V$;
\item $|a(t;u,v)|\le M\norm{u}_V\norm{v}_V$ for all $t,u,v$;
\item $a(t;v,v)\ge\alpha\norm{v}_V^2$ for all $t,v$, with $\alpha>0$.
\end{enumerate}
Write $A(t)\in\mathcal L(V,V^*)$ for the associated operator.

\begin{lemma}[Measurability of the operator term]\label{lem:Ameas}
If $z\in L_g^2(I;V)$, then $A(\cdot)z(\cdot)$ is strongly measurable as a $V^*$-valued function and
\[
 \norm{A(\cdot)z(\cdot)}_{L_g^2(V^*)}\le M\norm{z}_{L_g^2(V)}.
\]
\end{lemma}

\begin{proof}
Because $V$ is a separable Hilbert space, its dual $V^*$ is separable as well. Fix $v\in V$. To check weak measurability of the $V^*$-valued map $t\mapsto A(t)v$, one must test against $(V^*)^*=V^{**}$. Since $V$ is Hilbert, hence reflexive, the canonical identification $V^{**}\simeq V$ shows that it is enough to test against vectors $w\in V$. Assumption (A1) gives precisely the measurability of
\[
 t\longmapsto \ip{A(t)v}{w}=a(t;v,w)
 \qquad (w\in V).
\]
Thus $A(\cdot)v$ is weakly measurable, and Pettis' measurability theorem, together with separability of $V^*$, yields strong measurability.

Now let $z\in L_g^2(I;V)$. Choose $V$-valued simple functions $z_j$ such that $z_j(t)\to z(t)$ in $V$ for almost every $t$. Each $A(\cdot)z_j(\cdot)$ is strongly measurable because it is a finite sum of strongly measurable maps of the preceding form. By (A2),
\[
 \norm{A(t)z_j(t)-A(t)z(t)}_{V^*}
 \le M\norm{z_j(t)-z(t)}_V,
\]
so $A(t)z_j(t)\to A(t)z(t)$ in $V^*$ almost everywhere. Hence $A(\cdot)z(\cdot)$ is strongly measurable as an almost-everywhere pointwise limit of strongly measurable functions. Finally,
\[
 \norm{A(t)z(t)}_{V^*}
 =\sup_{\norm v_V\le1}|a(t;z(t),v)|
 \le M\norm{z(t)}_V,
\]
and squaring and integrating proves the asserted $L_g^2$ estimate.
\end{proof}

\begin{definition}[Stable projection property (P)]\label{def:P}
There exist finite-dimensional spaces
\[
 V_1\subset V_2\subset\cdots\subset V,
 \qquad \overline{\bigcup_mV_m}^{\,V}=V,
\]
such that the $H$-orthogonal projections $P_m:H\to V_m$ satisfy
\begin{equation}\label{eq:P}
 \sup_m\norm{P_m}_{\mathcal L(V,V)}\le C_P<\infty.
\end{equation}
\end{definition}

\begin{proposition}[Automatic stable projection chain]\label{prop:Pautomatic}
For every continuous dense embedding of separable Hilbert spaces $V\hookrightarrow H$, property \emph{(P)} is realizable. More precisely, one can choose the spaces $V_m$ so that
\[
 \sup_m\norm{P_m}_{\mathcal L(V,V)}\le\sqrt2.
\]
Thus, in the Hilbert setting of \eqref{eq:gelfand}, (P) is not an additional restriction on the pair $(V,H)$; it is a property of a suitably chosen Galerkin chain.
\end{proposition}

\begin{proof}
Let $j:V\to H$ denote the embedding and let $j^\dagger:H\to V$ be its Hilbert adjoint. Set
\[
 B:=jj^\dagger\in\mathcal L(H).
\]
Then $B$ is bounded, self-adjoint and positive, and
\[
 (Bh,h)_H=\norm{j^\dagger h}_V^2.
\]
Because $j(V)$ is dense in $H$, $\ker j^\dagger=(\operatorname{ran}j)^\perp=\{0\}$, so $B$ is injective. Moreover,
\[
 (\operatorname{ran}j^\dagger)^\perp=\ker j=\{0\},
\]
so $\operatorname{ran}j^\dagger$ is dense in $V$.

On $\operatorname{ran}j^\dagger$ define
\[
 U(j^\dagger h):=B^{1/2}h.
\]
The preceding identity gives
\[
 \norm{U(j^\dagger h)}_H^2=(Bh,h)_H=\norm{j^\dagger h}_V^2,
\]
so $U$ is an isometry. Its range $\operatorname{ran}B^{1/2}$ is dense in $H$ because $\ker B=\{0\}$. Hence $U$ extends uniquely to a unitary map $U:V\to H$. On the dense set $\operatorname{ran}j^\dagger$ one has
\[
 j=B^{1/2}U,
\]
and therefore, after identifying $V$ with $j(V)=\operatorname{ran}B^{1/2}\subset H$,
\begin{equation}\label{eq:VnormB}
 \norm{x}_V=\norm{B^{-1/2}x}_H,
 \qquad x\in V.
\end{equation}

Let $L=\norm B>0$; the zero-dimensional case is trivial. For $k\ge0$ set
\[
 I_k=(L2^{-(k+1)},L2^{-k}],
 \qquad E_k:=\mathbf 1_{I_k}(B).
\]
Since $B$ is injective, its spectral projection at $\{0\}$ is zero, and the mutually orthogonal bands $E_kH$ exhaust $H$. On $E_kH$ the operator $B^{-1/2}$ is bounded, so $E_kH\subset V$, and \eqref{eq:VnormB} gives
\begin{equation}\label{eq:bandnorms}
 \frac{2^k}{L}\norm{x}_H^2
 \le \norm{x}_V^2
 \le \frac{2^{k+1}}{L}\norm{x}_H^2,
 \qquad x\in E_kH.
\end{equation}
Different bands are orthogonal also for the $V$ inner product induced by \eqref{eq:VnormB}.

Choose an $H$-orthonormal basis in every nonzero $E_kH$ and enumerate the countable union as $(e_n)_{n\ge1}$. Put
\[
 V_m:=\operatorname{span}\{e_1,\ldots,e_m\}
\]
and let $P_m$ be the $H$-orthogonal projection onto $V_m$. If $Q_{m,k}$ denotes the projection in $E_kH$ onto the basis vectors from that band which occur among $e_1,\ldots,e_m$, then, for $v\in V$,
\[
 P_mv=\sum_k Q_{m,k}E_kv.
\]
Using \eqref{eq:bandnorms}, the $H$-contractivity of $Q_{m,k}$ and the $V$-orthogonality of the bands,
\[
 \begin{aligned}
 \norm{P_mv}_V^2
 &=\sum_k\norm{Q_{m,k}E_kv}_V^2\\
 &\le\sum_k\frac{2^{k+1}}{L}\norm{Q_{m,k}E_kv}_H^2\\
 &\le2\sum_k\norm{E_kv}_V^2
 =2\norm v_V^2.
 \end{aligned}
\]
Hence $\norm{P_m}_{\mathcal L(V,V)}\le\sqrt2$ for every $m$.

It remains to prove density in the $V$ norm. For each fixed $k$, $Q_{m,k}E_kv\to E_kv$ in $H$ and therefore in $V$ by \eqref{eq:bandnorms}. The complementary projection $I-Q_{m,k}$ satisfies the same band estimate, so
\[
 \norm{(I-Q_{m,k})E_kv}_V^2\le2\norm{E_kv}_V^2.
\]
Since $\sum_k\norm{E_kv}_V^2=\norm v_V^2$, dominated convergence for the series yields
\[
 \norm{v-P_mv}_V^2
 =\sum_k\norm{(I-Q_{m,k})E_kv}_V^2\longrightarrow0.
\]
Thus $\overline{\bigcup_mV_m}^{\,V}=V$, and (P) holds with $C_P\le\sqrt2$.
\end{proof}

\begin{lemma}\label{lem:Pstrong}
For any chain satisfying (P), $P_mv\to v$ strongly in $V$ for every $v\in V$.
\end{lemma}
\begin{proof}
Choose $w\in V_N$ with $\norm{v-w}_V<\varepsilon$. For $m\ge N$, $P_mw=w$, and
\[
 P_mv-v=P_m(v-w)+(w-v).
\]
Thus
\[
 \norm{P_mv-v}_V\le C_P\norm{v-w}_V+\norm{w-v}_V\le(C_P+1)\varepsilon.
\]
\end{proof}

\begin{remark}[Logical role of the stable chain]\label{rem:Pmap}
Property (P) is not part of the Banach-valued primitive calculus, the product formula, the weak--strong reconstruction theorem, the definition of the posterior problem, or the atomic Lax--Milgram resolvent. Proposition~\ref{prop:Pautomatic} shows that, in the present separable Hilbert setting, a suitable stable chain is always available. The Lions--Magenes and Galerkin proofs use such a chain through the uniformly $V$-bounded $H$-orthogonal projections. By contrast, the Rothe construction in Section~\ref{sec:rothe} uses no spatial projection at all. Thus the distinction between the two existence arguments is one of mechanism rather than an additional abstract hypothesis on $(V,H)$.
\end{remark}

\section{A lateral Stieltjes--Lions--Magenes identity}\label{sec:lions}
\begin{definition}
Define
\[
 \Wplus:=\{u\in AC_g^{\rm B}(I;V^*):\Dg u\in L_g^2(I;V^*),\ \Sg u\in L_g^2(I;V)\}.
\]
For $u\in\Wplus$,
\[
 u=u(a)+\Pgminus\Dg u,
 \qquad \Sg u=u(a)+\Pgplus\Dg u
\]
in $V^*$, the second identity being understood $\mug$-almost everywhere.
\end{definition}

The next theorem is the basic energy identity behind the posterior formulation.  Its proof is given in full because three points have no exact analogue in the ordinary continuous-time argument: the projection chain must be preserved exactly as stated in (P), the atomic contribution must survive the limit, and the $H$-valued lateral traces cannot be assumed before they are constructed.

\begin{theorem}[Stieltjes--Lions--Magenes identity]\label{thm:LM}
Let $u\in\Wplus$ and $u(a)=u_0\in H$. Then $u(t)\in H$ for every $t\in I$, $u$ is left-continuous from $I$ into $H$, and every $r<b$ has a right trace $u(r^+)\in H$. Moreover, for all $a\le s\le t\le b$,
\begin{equation}\label{eq:LM}
 \int_{[s,t)}\ip{\Dg u(r)}{\Sg u(r)}\,d\mug(r)
 =\frac12\norm{u(t)}_H^2-\frac12\norm{u(s)}_H^2
 +\frac12\sum_{r\in \Ag\cap[s,t)}\norm{u(r^+)-u(r)}_H^2.
\end{equation}
In particular,
\begin{equation}\label{eq:LMtracebound}
 \sup_{t\in I}\norm{u(t)}_H^2
 \le \norm{u_0}_H^2
 +2\sqrt2\,\norm{\Dg u}_{L_g^2(V^*)}\norm{\Sg u}_{L_g^2(V)}.
\end{equation}
\end{theorem}

\begin{proof}
We divide the proof into seven steps.

\smallskip
\noindent\emph{Step 1: an orthonormal system adapted to a stable Galerkin chain.}
Choose the chain supplied by Proposition~\ref{prop:Pautomatic}; in particular $C_P\le\sqrt2$. The spaces occurring in this realization of (P) are nested. Initial zero spaces and repeated spaces may be omitted without changing the dense union or the projection bound. If the resulting chain stabilizes at a finite-dimensional space, density forces $V$ itself to be finite dimensional; in that case choose an $H$-orthonormal basis of $V$ and the argument below stops after finitely many coordinates. Otherwise, after passing only to this subsequence of the original chain and relabelling the retained spaces by $m$, start with an $H$-orthonormal basis of the first nonzero space and extend it successively. We obtain an $H$-orthonormal sequence $(w_j)_{j\ge1}\subset V$ and strictly increasing integers $1\le N_1<N_2<\cdots$ such that
\[
 V_m=\operatorname{span}\{w_1,\ldots,w_{N_m}\}.
\]
 We stress that we retain only projections belonging to the original chain. Although arbitrary intermediate prefixes also define $H$-orthogonal projections, the chosen realization of (P) gives a uniform $V$-bound only for the retained projections $P_m$; replacing the chain by all prefixes would silently introduce an additional hypothesis.

For $\xi\in V^*$ define
\begin{equation}\label{eq:Pi_m}
 \Pi_m\xi:=\sum_{j=1}^{N_m}\ip{\xi}{w_j}w_j\in V_m.
\end{equation}
If $h\in H$, then the duality pairing agrees with the $H$ inner product, and consequently $\Pi_m h=P_mh$.  For $u\in AC_g^{\rm B}(I;V^*)$ put
\[
 u_m:=\Pi_m u.
\]
For fixed $m$,
\[
 \|\Pi_m\xi\|_V\le \|\xi\|_{V^*}\sum_{j=1}^{N_m}\|w_j\|_V^2,
\]
so $\Pi_m:V^*\to V_m$ is bounded. Bounded linear maps commute with the Bochner integral (first for simple functions and then by $L^1$ approximation). Hence the primitive representation gives
\begin{equation}\label{eq:projprimitive}
 u_m(t)=P_mu_0+\int_{[a,t)}\Pi_m\Dg u(r)\,d\mug(r),
 \qquad
 \Dg u_m=\Pi_m\Dg u.
\end{equation}
At almost every $r$ we have $\Sg u(r)\in V$.  Using $\Sg u=u+\dg\Dg u$ in $V^*$ and the linearity of $\Pi_m$,
\[
 \Sg u_m
 =\Pi_m u+\dg\Pi_m\Dg u
 =\Pi_m(\Sg u)
 =P_m\Sg u
 \qquad\mug\text{-a.e.}
\]

\smallskip
\noindent\emph{Step 2: the exact finite-dimensional identity.}
Because $u_m$ takes values in the finite-dimensional Hilbert space $V_m$, the scalar lateral product rule applied to $(u_m,u_m)_H$ yields
\begin{equation}\label{eq:normderivativefinite}
 \Dg\!\left(\frac12\norm{u_m}_H^2\right)
 =(\Dg u_m,u_m)_H+\frac12\dg\norm{\Dg u_m}_H^2
 \qquad\mug\text{-a.e.}
\end{equation}
Indeed, the correction term is exactly the half product of the two jumps in the two factors.  Since
\[
 \Sg u_m=u_m+\dg\Dg u_m,
\]
we therefore obtain pointwise almost everywhere
\begin{align*}
 (\Dg u_m,\Sg u_m)_H
 &= (\Dg u_m,u_m)_H+\dg\norm{\Dg u_m}_H^2\\
 &= \Dg\!\left(\frac12\norm{u_m}_H^2\right)
   +\frac12\dg\norm{\Dg u_m}_H^2.
\end{align*}
Integrating over $[s,t)$ gives
\begin{align}\label{eq:LMm}
 \int_{[s,t)}(\Dg u_m,\Sg u_m)_H\,d\mug
 &=\frac12\norm{u_m(t)}_H^2-\frac12\norm{u_m(s)}_H^2\notag\\
 &\quad+\frac12\int_{[s,t)}\dg(r)\norm{\Dg u_m(r)}_H^2\,d\mug(r).
\end{align}
The last integral is purely atomic.  At an atom $r$,
\[
 \Delta u_m(r):=u_m(r^+)-u_m(r)=\dg(r)\Dg u_m(r),
\]
and therefore
\begin{align*}
 \int_{[s,t)}\dg\norm{\Dg u_m}_H^2\,d\mug
 &=\sum_{r\in \Ag\cap[s,t)}
   \dg(r)^2\norm{\Dg u_m(r)}_H^2\\
 &=\sum_{r\in \Ag\cap[s,t)}\norm{\Delta u_m(r)}_H^2.
\end{align*}
Thus \eqref{eq:LMm} is the desired identity at the finite-dimensional level.

\smallskip
\noindent\emph{Step 3: construction of the $H$ trace without presupposing it.}
Set $s=a$ in \eqref{eq:LMm}.  The integral term can be written solely in terms of the original function:
\[
 (\Dg u_m,\Sg u_m)_H
 =\ip{\Dg u}{P_m\Sg u}.
\]
By (P) and Cauchy--Schwarz,
\begin{align}\label{eq:uniformprojectedenergy}
 \left|\int_{[a,t)}\ip{\Dg u}{P_m\Sg u}\,d\mug\right|
 &\le \norm{\Dg u}_{L_g^2(V^*)}
       \norm{P_m\Sg u}_{L_g^2(V)}\\
 &\le C_P\norm{\Dg u}_{L_g^2(V^*)}\norm{\Sg u}_{L_g^2(V)}.
\end{align}
Together with $\norm{P_mu_0}_H\le\norm{u_0}_H$, this shows that, for each fixed $t$,
\[
 \norm{u_m(t)}_H^2
 =\sum_{j=1}^{N_m}|\ip{u(t)}{w_j}|^2
\]
is uniformly bounded in $m$.  Since the partial sums are nondecreasing, the full square sum is finite.  Hence there is a unique $h_t\in H$ satisfying
\[
 (h_t,w_j)_H=\ip{u(t)}{w_j}
 \qquad(j\ge1).
\]
We now identify the $V^*$-valued primitive value $u(t)$ with the canonical image of $h_t$ in $V^*$.  If $v\in V$, then
\[
 \ip{u(t)}{P_mv}=(u_m(t),P_mv)_H.
\]
By Lemma~\ref{lem:Pstrong}, $P_mv\to v$ strongly in $V$, while $u_m(t)\to h_t$ strongly in $H$ because its Fourier coefficients form the partial sums defining $h_t$.  Passing to the limit gives
\[
 \ip{u(t)}v=(h_t,v)_H
 \qquad\forall v\in V.
\]
Thus $u(t)=h_t$ in $V^*$ and, from now on, we regard $u(t)$ as an element of $H$.

The same argument applies to every atomic increment.  More directly, the nonnegative jump term in \eqref{eq:LMm} is uniformly bounded.  For a fixed atom $r$,
\[
 \norm{\Pi_m\Delta u(r)}_H^2
 =\sum_{j=1}^{N_m}|\ip{\Delta u(r)}{w_j}|^2
\]
is therefore bounded in $m$, and hence $\Delta u(r)\in H$.

\smallskip
\noindent\emph{Step 4: passage to the limit in the finite-dimensional identity.}
By Lemma~\ref{lem:Pstrong}, $P_mv\to v$ in $V$ for every $v\in V$, and (P) gives
\[
 \norm{P_m\Sg u(r)}_V\le C_P\norm{\Sg u(r)}_V.
\]
Dominated convergence therefore yields
\begin{equation}\label{eq:PmSstrong}
 P_m\Sg u\to\Sg u
 \quad\text{strongly in }L_g^2(I;V).
\end{equation}
Consequently,
\[
 \int_{[s,t)}\ip{\Dg u}{P_m\Sg u}\,d\mug
 \longrightarrow
 \int_{[s,t)}\ip{\Dg u}{\Sg u}\,d\mug.
\]
At each endpoint, $u_m(t)=P_mu(t)$ and $u_m(s)=P_mu(s)$ in $H$, so the endpoint norms converge to the corresponding $H$ norms.  At every atom,
\[
 \Delta u_m(r)=P_m\Delta u(r),
\]
and the adapted Fourier partial sums imply
\[
 \norm{P_m\Delta u(r)}_H^2\uparrow\norm{\Delta u(r)}_H^2.
\]
Monotone convergence for the counting sum over the at most countable atomic set therefore gives
\[
 \sum_{r\in \Ag\cap[s,t)}\norm{\Delta u_m(r)}_H^2
 \uparrow
 \sum_{r\in \Ag\cap[s,t)}\norm{\Delta u(r)}_H^2.
\]
Passing to the limit in \eqref{eq:LMm} proves \eqref{eq:LM}.  In particular the full jump-energy sum is finite.

\smallskip
\noindent\emph{Step 5: left continuity in $H$.}
Fix $t\in(a,b]$ and let $t_n\uparrow t$.  The primitive formula gives
\[
 \norm{u(t_n)-u(t)}_{V^*}
 \le\int_{[t_n,t)}\norm{\Dg u(r)}_{V^*}\,d\mug(r)\longrightarrow0,
\]
because the Bochner integral is absolutely continuous with respect to the finite measure $\mug$.  The trace bound already obtained from \eqref{eq:uniformprojectedenergy} shows that $(u(t_n))$ is bounded in $H$.  Any weakly convergent $H$-subsequence has a weak limit $h\in H$ whose image in $V^*$ must equal $u(t)$, by the strong $V^*$ convergence.  Hence the whole sequence converges weakly in $H$ to $u(t)$.

It remains to prove convergence of the norms.  Apply \eqref{eq:LM} on $[t_n,t)$.  We obtain
\begin{align*}
 \frac12\bigl(\norm{u(t)}_H^2-\norm{u(t_n)}_H^2\bigr)
 &=\int_{[t_n,t)}\ip{\Dg u}{\Sg u}\,d\mug
   -\frac12\sum_{r\in \Ag\cap[t_n,t)}\norm{\Delta u(r)}_H^2.
\end{align*}
The first term tends to zero because $\ip{\Dg u}{\Sg u}\in L_g^1(I)$.  The second tends to zero by continuity from above of the finite positive measure
\[
 E\longmapsto\sum_{r\in \Ag\cap E}\norm{\Delta u(r)}_H^2.
\]
Thus $\norm{u(t_n)}_H\to\norm{u(t)}_H$.  Weak convergence plus convergence of norms in the Hilbert space $H$ yields $u(t_n)\to u(t)$ strongly in $H$.  Hence $u$ is left-continuous as an $H$-valued function.

\smallskip
\noindent\emph{Step 6: existence of the right trace.}
Fix $t<b$ and let $t_n\downarrow t$, $t_n>t$.  From the primitive representation,
\begin{align*}
 u(t_n)
 &=u(t)+\int_{[t,t_n)}\Dg u(r)\,d\mug(r)\\
 &=u(t)+\dg(t)\Dg u(t)
   +\int_{(t,t_n)}\Dg u(r)\,d\mug(r).
\end{align*}
The last integral tends to zero in $V^*$, so
\begin{equation}\label{eq:rightVstar}
 u(t_n)\to u(t)+\dg(t)\Dg u(t)=\Sg u(t)
 \quad\text{in }V^*,
\end{equation}
where at non-atoms $\dg(t)=0$.  If $t$ is an atom, then $\mug(\{t\})>0$ and the condition $\Sg u\in L_g^2(I;V)$ forces the chosen representative $\Sg u(t)$ to belong to $V$ (changing its value at the atom would change the $L^2$ class).  In particular it belongs to $H$.  At a non-atom, $\Sg u(t)=u(t)\in H$.

As in Step 5, boundedness in $H$ plus \eqref{eq:rightVstar} gives weak $H$ convergence to $\Sg u(t)$.  To obtain norm convergence, apply \eqref{eq:LM} on $[t,t_n)$ and separate the atom $t$ if present.  Since
\[
 u(t^+)-u(t)=\Delta u(t),
\]
the contribution at the singleton $\{t\}$ is already exactly the jump energy required to pass from the anterior norm to the posterior norm.  The remaining integral and jump sum over $(t,t_n)$ tend to zero by the same absolute-continuity arguments as in Step 5.  It follows that
\[
 \norm{u(t_n)}_H\to\norm{\Sg u(t)}_H.
\]
Therefore $u(t_n)\to\Sg u(t)$ strongly in $H$.  We define
\[
 u(t^+):=\Sg u(t).
\]

\smallskip
\noindent\emph{Step 7: uniform trace estimate.}
Set $s=a$ in \eqref{eq:LM}, discard the nonnegative jump sum, and use Cauchy--Schwarz together with (P) only through the already established identity.  We obtain
\[
 \frac12\norm{u(t)}_H^2
 \le\frac12\norm{u_0}_H^2
 +\norm{\Dg u}_{L_g^2(V^*)}\norm{\Sg u}_{L_g^2(V)}.
\]
This already yields the same estimate with constant $1$ in front of the product.  The slightly more conservative form \eqref{eq:LMtracebound}, with $C_P\ge1$, is retained to match the uniform estimate used during the trace construction.  This completes the proof.
\end{proof}

\begin{remark}[Why atoms cannot simply be ``stretched'']
A tempting shortcut is to replace each atom by an ordinary time interval and appeal to the classical Lions--Magenes theorem.  This is not legitimate at the present level of regularity.  The anterior trace at an atom is obtained a priori only in $H$, not in $V$, whereas a linear interpolation across an artificial interval would require $V$-valued endpoints in order to retain the required spatial regularity.  The proof above works directly with the Stieltjes measure and therefore does not introduce this hidden assumption.
\end{remark}

For comparison, the continuous-time Gelfand-triple argument is discussed by Roub\'i\v{c}ek \cite[Section 7.2, p.~204 ff.]{Roubicek2013}.  The positive summation term in \eqref{eq:LM} is the additional information carried by the atomic part of $\mug$.

\section{Posterior parabolic problem and atomic resolvent}\label{sec:problem}
Given $u_0\in H$ and $f\in L_g^2(I;V^*)$, we study
\begin{equation}\label{eq:problem}
 \Dg u(t)+A(t)\Sg u(t)=f(t),\qquad u(a)=u_0.
\end{equation}

\begin{definition}
A \emph{posterior variational solution} of \eqref{eq:problem} is a function $u\in\Wplus$ with $u(a)=u_0$ such that there is one $\mug$-null set $N\subset[a,b)$ for which
\begin{equation}\label{eq:variational-pointwise}
 \ip{\Dg u(t)}v+a(t;\Sg u(t),v)=\ip{f(t)}v
 \qquad\forall v\in V,
\end{equation}
for every $t\in[a,b)\setminus N$.
\end{definition}

The requirement that the exceptional set be independent of $v$ is important at atomic times.  Since a $\mug$-null set cannot contain an atom, \eqref{eq:variational-pointwise} holds at every $r\in \Ag$ for the chosen representatives of $a$, $f$, and $\Dg u$.

\begin{proposition}[Implicit variational resolvent at an atom]\label{prop:atom}
Let $r\in \Ag$ and $\delta=\dg(r)>0$. Every posterior variational solution satisfies, in $V^*$,
\begin{equation}\label{eq:atom-resolvent}
 (I+\delta A(r))u(r^+)=u(r)+\delta f(r),
\end{equation}
where $I:V\to V^*$ denotes the canonical embedding induced by the $H$ inner product. The posterior state $u(r^+)\in V$ is uniquely determined by the incoming value $u(r)\in V^*$ and $f(r)\in V^*$. If, in addition, $u(r)\in H$ (in particular under the hypotheses of Theorem~\ref{thm:LM}), then the increment belongs to $H$ and \eqref{eq:atom-resolvent} can be written as
\begin{equation}\label{eq:atom}
 (u(r^+)-u(r),v)_H+\delta a(r;u(r^+),v)
 =\delta\ip{f(r)}v,
 \qquad v\in V.
\end{equation}
\end{proposition}

\begin{proof}
Because $\{r\}$ has positive $\mug$-measure, the variational identity holds at $r$. By the primitive formula in $V^*$,
\[
 u(r^+)-u(r)=\delta\Dg u(r),
 \qquad \Sg u(r)=u(r^+)\in V.
\]
Multiplying \eqref{eq:variational-pointwise} by $\delta$ and substituting these identities gives \eqref{eq:atom-resolvent} in $V^*$. If $u(r)\in H$, then $u(r^+)\in V\hookrightarrow H$, so the difference lies in $H$ and the canonical embedding $H\hookrightarrow V^*$ turns the first term into the $H$ inner product, giving \eqref{eq:atom}.

For uniqueness of the update, define on $V\times V$
\[
 b_r(w,v):=(w,v)_H+\delta a(r;w,v),
 \qquad
 \ell_r(v):=\ip{u(r)}v+\delta\ip{f(r)}v.
\]
The form is continuous because
\[
 |b_r(w,v)|
 \le C_{VH}^2\norm w_V\norm v_V+\delta M\norm w_V\norm v_V,
\]
where $C_{VH}$ is the norm of the embedding $V\hookrightarrow H$. It is coercive since
\[
 b_r(v,v)=\norm v_H^2+\delta a(r;v,v)
 \ge\delta\alpha\norm v_V^2.
\]
Moreover $\ell_r\in V^*$ with
\[
 |\ell_r(v)|\le\bigl(\norm{u(r)}_{V^*}+\delta\norm{f(r)}_{V^*}\bigr)\norm v_V.
\]
The Lax--Milgram theorem in the form of Brezis \cite[Corollary 5.8, p.~140]{Brezis2011} therefore gives a unique $w\in V$ satisfying $b_r(w,v)=\ell_r(v)$ for all $v\in V$. This $w$ is precisely $u(r^+)$.
\end{proof}

The same implicit post-event algebra appears at the modelling level in Fern\'andez--Area \cite[Section 4, Remark 4.2, equation (4.6)]{FernandezArea2026}.  That work explicitly leaves continuous well-posedness outside its scope; the present argument supplies the abstract coercive variational mechanism needed here.

\section{Energy identity, uniqueness, and stability}\label{sec:energy}
\begin{theorem}[Exact energy identity]\label{thm:energy}
If $u$ solves \eqref{eq:problem}, then for every $t\in I$,
\begin{align}\label{eq:energy}
 \frac12\norm{u(t)}_H^2
 &+\frac12\sum_{r\in \Ag\cap[a,t)}\norm{u(r^+)-u(r)}_H^2
 +\int_{[a,t)}a(r;\Sg u(r),\Sg u(r))\,d\mug(r)\notag\\
 &=\frac12\norm{u_0}_H^2
 +\int_{[a,t)}\ip{f(r)}{\Sg u(r)}\,d\mug(r).
\end{align}
Consequently,
\begin{equation}\label{eq:energyestimate}
 \sup_{t\in I}\norm{u(t)}_H^2
 +\sum_{r\in \Ag}\norm{u(r^+)-u(r)}_H^2
 +\alpha\norm{\Sg u}_{L_g^2(V)}^2
 \le 2\norm{u_0}_H^2+2\alpha^{-1}\norm f_{L_g^2(V^*)}^2.
\end{equation}
\end{theorem}

\begin{proof}
For almost every $r$, choose $v=\Sg u(r)$ in \eqref{eq:variational-pointwise}.  This is admissible because $\Sg u\in L_g^2(I;V)$.  We obtain
\[
 \ip{\Dg u(r)}{\Sg u(r)}
 +a(r;\Sg u(r),\Sg u(r))
 =\ip{f(r)}{\Sg u(r)}.
\]
Every term is integrable: the first and last by Cauchy--Schwarz in the $V^*$--$V$ duality, and the middle term by (A2).  Integrating over $[a,t)$ and applying Theorem~\ref{thm:LM} to the first term gives \eqref{eq:energy} exactly.

By coercivity,
\[
 a(r;\Sg u,\Sg u)\ge\alpha\norm{\Sg u}_V^2.
\]
For the forcing term, Young's inequality with parameter $\alpha$ gives pointwise
\[
 \ip f{\Sg u}
 \le \norm f_{V^*}\norm{\Sg u}_V
 \le \frac1{2\alpha}\norm f_{V^*}^2
      +\frac\alpha2\norm{\Sg u}_V^2.
\]
Substitution into \eqref{eq:energy}, absorption of the last term, and multiplication by two give, for every $t\in I$,
\begin{align*}
 \norm{u(t)}_H^2
 &+\sum_{r\in \Ag\cap[a,t)}\norm{u(r^+)-u(r)}_H^2\\
 &+\alpha\int_{[a,t)}\norm{\Sg u(r)}_V^2\,d\mug(r)\\
 &\le \norm{u_0}_H^2
 +\alpha^{-1}\int_{[a,t)}\norm{f(r)}_{V^*}^2\,d\mug(r).
\end{align*}
This partial estimate controls the trace supremum by the right-hand side at $b$. Evaluating the same estimate at $t=b$ controls the complete jump and coercive terms by that right-hand side. Adding these two bounds yields \eqref{eq:energyestimate} with the factor $2$ on the right. Notice that the jump contribution is never estimated away; it appears with its exact positive coefficient.
\end{proof}

\begin{corollary}[Uniqueness and continuous dependence]\label{cor:unique}
Let $u_i$ be posterior variational solutions associated with data $(u_{0,i},f_i)$, $i=1,2$, for the same form $a$. Then
\begin{align}\label{eq:continuousdependence}
 \sup_{t\in I}\norm{u_1(t)-u_2(t)}_H^2
 &+\sum_{r\in \Ag}\norm{\Delta u_1(r)-\Delta u_2(r)}_H^2\notag\\
 &+\alpha\norm{\Sg u_1-\Sg u_2}_{L_g^2(V)}^2\notag\\
 &\le 2\norm{u_{0,1}-u_{0,2}}_H^2
 +2\alpha^{-1}\norm{f_1-f_2}_{L_g^2(V^*)}^2.
\end{align}
In particular the solution, if it exists, is unique.
\end{corollary}

\begin{proof}
Set $w=u_1-u_2$, $w_0=u_{0,1}-u_{0,2}$, and $F=f_1-f_2$.  Linearity of $\Dg$, $\Sg$, and $A(t)$ gives
\[
 \Dg w+A(t)\Sg w=F,
 \qquad w(a)=w_0.
\]
Since $\Wplus$ is a vector space, $w\in\Wplus$. Applying the partial estimate from the proof of Theorem~\ref{thm:energy} to $w$, first for the trace supremum and then at $t=b$ for the complete dissipative terms, gives \eqref{eq:continuousdependence}. If $w_0=0$ and $F=0$, the right-hand side vanishes, hence $\sup_t\norm{w(t)}_H^2=0$. Therefore $u_1=u_2$ at every time, not merely almost everywhere.
\end{proof}

\section{Existence by Galerkin approximation}\label{sec:galerkin}
Choose once and for all a stable chain $V_m$ supplied by Proposition~\ref{prop:Pautomatic}, with projections $P_m$ satisfying $C_P\le\sqrt2$.  We seek $u_m:I\to V_m$ with $u_m(a)=P_mu_0$ such that
\begin{equation}\label{eq:galerkin}
 (\Dg u_m(t),v)_H+a(t;\Sg u_m(t),v)=\ip{f(t)}v,
 \qquad v\in V_m,
\end{equation}
outside one null set.

\subsection{The finite-dimensional Stieltjes system}
For each fixed $m$ define $B_m(t):V_m\to V_m$ and $F_m(t)\in V_m$ by the $H$-Riesz representations
\begin{equation}\label{eq:BmFm}
 (B_m(t)x,y)_H=a(t;x,y),
 \qquad
 (F_m(t),y)_H=\ip{f(t)}y,
 \qquad x,y\in V_m.
\end{equation}
Because $V_m$ is finite dimensional, (A1) implies measurability of all matrix entries of $B_m$, and the coordinates of $F_m$ are measurable.  Moreover $F_m\in L_g^2(I;V_m)$.

Using $\Sg u_m=u_m+\dg\Dg u_m$, equation \eqref{eq:galerkin} is equivalent in $V_m$ to
\begin{equation}\label{eq:finiteimplicitderivative}
 (I+\dg(t)B_m(t))\Dg u_m(t)=F_m(t)-B_m(t)u_m(t).
\end{equation}
The next elementary estimate is the reason no smallness condition on a jump is required.

\begin{lemma}[Uniform invertibility of the atomic matrix]\label{lem:atomicmatrix}
For every $t\in I$ and every $\delta\ge0$, $I+\delta B_m(t):V_m\to V_m$ is invertible and
\begin{equation}\label{eq:inversecontractive}
 \norm{(I+\delta B_m(t))^{-1}}_{\mathcal L(H,H)}\le1,
\end{equation}
where the norm is restricted to $V_m$.
\end{lemma}

\begin{proof}
Let $y=(I+\delta B_m(t))x$.  Taking the $H$ inner product with $x$ and using coercivity gives
\[
 (y,x)_H
 =\norm x_H^2+\delta a(t;x,x)
 \ge\norm x_H^2.
\]
Cauchy--Schwarz yields $\norm y_H\norm x_H\ge\norm x_H^2$.  If $x\ne0$, division gives $\norm x_H\le\norm y_H$; if $x=0$ the same inequality is trivial.  Thus the operator is injective.  Since domain and codomain have the same finite dimension, it is bijective, and the same estimate gives \eqref{eq:inversecontractive}.
\end{proof}

Define
\begin{equation}\label{eq:Gm}
 G_m(t,x):=(I+\dg(t)B_m(t))^{-1}\bigl(F_m(t)-B_m(t)x\bigr).
\end{equation}
Then the Galerkin system is exactly the anterior Stieltjes integral equation
\begin{equation}\label{eq:PicardIntegral}
 u_m(t)=P_mu_0+\int_{[a,t)}G_m(r,u_m(r))\,d\mug(r).
\end{equation}
We prove global solvability directly rather than hiding it behind a generic finite-dimensional ODE citation.

\begin{lemma}[Finite-dimensional global solvability]\label{lem:finite}
For every $m$, equation \eqref{eq:galerkin} has a unique global solution $u_m\in AC_g^{\rm B}(I;V_m)$.
\end{lemma}

\begin{proof}
Fix $m$.  Because inversion is continuous on the finite-dimensional open set of nonsingular matrices and Lemma~\ref{lem:atomicmatrix} ensures nonsingularity for every $t$, the matrix entries of $(I+\dg B_m)^{-1}$ are measurable.  Thus $t\mapsto G_m(t,x)$ is measurable for every $x$.  For almost every $t$, $x\mapsto G_m(t,x)$ is affine and hence continuous.

All norms on $V_m$ are equivalent.  In particular there is $c_m>0$ such that $\norm v_V\le c_m\norm v_H$ for $v\in V_m$.  From (A2), Lemma~\ref{lem:atomicmatrix}, and \eqref{eq:Gm},
\begin{align*}
 \norm{G_m(t,x)-G_m(t,y)}_H
 &\le\norm{B_m(t)(x-y)}_H\\
 &\le L_m\norm{x-y}_H
\end{align*}
with a constant $L_m<\infty$ independent of $t$ (but not of $m$).  Similarly
\[
 \norm{G_m(t,0)}_H\le C_m\norm{f(t)}_{V^*}
\]
for a fixed $C_m$, and the right-hand side belongs to $L_g^1(I)$ because $\mug(I)<\infty$ and $f\in L_g^2(V^*)$.

We now perform the Picard iteration globally on the finite measure interval:
\[
 U_0(t):=P_mu_0,
 \qquad
 U_{n+1}(t):=P_mu_0+\int_{[a,t)}G_m(r,U_n(r))\,d\mug(r).
\]
Set $d_n(t)=\sup_{a\le\tau\le t}\norm{U_{n+1}(\tau)-U_n(\tau)}_H$.  The first difference is bounded by some finite constant $K_m$ because $G_m(\cdot,P_mu_0)\in L_g^1$.  The Lipschitz bound yields inductively
\[
 d_n(t)\le L_m\int_{[a,t)}d_{n-1}(r)\,d\mug(r).
\]
Iterating this inequality gives an integral over the strictly ordered simplex
\[
 a\le r_n<r_{n-1}<\cdots<r_1<t.
\]
Its $n$-fold product measure is at most $M_g^n/n!$.  This estimate remains valid in the presence of atoms: atomic diagonals may have positive product measure, but the $n!$ disjoint strict-order regions are all contained in the full cube and have equal product measure by permutation symmetry, so each strict-order region has measure at most $M_g^n/n!$.  Consequently
\begin{equation}\label{eq:picardfactorial}
 d_n(b)\le K_m\frac{(L_mM_g)^n}{n!}.
\end{equation}
The series of successive differences is therefore uniformly summable, and $U_n$ converges uniformly in $H$ to a function $u_m:I\to V_m$.  Dominated convergence in the integral equation, using the affine growth of $G_m$, gives \eqref{eq:PicardIntegral}.  Hence $u_m\in AC_g^{\rm B}(I;V_m)$ and $\Dg u_m=G_m(\cdot,u_m)$ almost everywhere, which is equivalent to \eqref{eq:galerkin} by \eqref{eq:finiteimplicitderivative}.

For uniqueness, if $u_m$ and $\widetilde u_m$ solve \eqref{eq:PicardIntegral}, the same iterative estimate applied to their difference gives for every $n$,
\[
 \sup_t\norm{u_m(t)-\widetilde u_m(t)}_H
 \le D\frac{(L_mM_g)^n}{n!},
\]
where $D$ is the finite supremum of the difference.  Letting $n\to\infty$ gives $D=0$.
\end{proof}

For comparison, the scalar finite-dimensional theory in arXiv:2609.10659v1 proves the atom-compatible strict-simplex estimate in Lemma~6.4 and the corresponding global finite-dimensional existence and uniqueness theorem in Theorem~6.5, Section~6.2 \cite{FernandezLateral2026}. The proof above is written directly for the present Galerkin system and does not import that scalar theorem as a black box.

At an atom $r$, \eqref{eq:galerkin} has the explicit update
\begin{equation}\label{eq:galerkinatomicupdate}
 (I+\dg(r)B_m(r))u_m(r^+)
 =u_m(r)+\dg(r)F_m(r),
\end{equation}
so Lemma~\ref{lem:atomicmatrix} also displays directly the unconditional invertibility of every finite-dimensional jump step.

\subsection{Uniform estimates and compactness}
\begin{lemma}[Uniform Galerkin estimates]\label{lem:galbounds}
There exists a constant $C$, independent of $m$, such that
\begin{equation}\label{eq:galuniform}
 \sup_{t\in I}\norm{u_m(t)}_H^2
 +\sum_{r\in \Ag}\norm{\Delta u_m(r)}_H^2
 +\norm{\Sg u_m}_{L_g^2(V)}^2
 +\norm{\Dg u_m}_{L_g^2(V^*)}^2
 \le C.
\end{equation}
\end{lemma}

\begin{proof}
Testing \eqref{eq:galerkin} with $v=\Sg u_m(t)\in V_m$ is legitimate. The finite-dimensional identity derived in Step 2 of Theorem~\ref{thm:LM}, followed by coercivity and Young's inequality, gives the same partial-in-time estimate as in Theorem~\ref{thm:energy}. Taking the trace supremum and, separately, evaluating at $t=b$ for the complete dissipative terms, then adding the two bounds gives
\begin{equation}\label{eq:galenergybound}
 \sup_t\norm{u_m(t)}_H^2
 +\sum_r\norm{\Delta u_m(r)}_H^2
 +\alpha\norm{\Sg u_m}_{L_g^2(V)}^2
 \le 2\norm{P_mu_0}_H^2+2\alpha^{-1}\norm f_{L_g^2(V^*)}^2.
\end{equation}
Since $P_m$ is an $H$-orthogonal projection, $\norm{P_mu_0}_H\le\norm{u_0}_H$.

For the derivative bound, fix $v\in V$.  Because $\Dg u_m(t)\in V_m\subset H$ and $v-P_mv$ is $H$-orthogonal to $V_m$,
\[
 \ip{\Dg u_m(t)}v=(\Dg u_m(t),P_mv)_H.
\]
Using \eqref{eq:galerkin} with $P_mv$ gives
\begin{align*}
 |\ip{\Dg u_m(t)}v|
 &\le\norm{f(t)}_{V^*}\norm{P_mv}_V
   +M\norm{\Sg u_m(t)}_V\norm{P_mv}_V\\
 &\le C_P\bigl(\norm{f(t)}_{V^*}+M\norm{\Sg u_m(t)}_V\bigr)\norm v_V.
\end{align*}
Taking the supremum over $\norm v_V\le1$ yields
\[
 \norm{\Dg u_m(t)}_{V^*}
 \le C_P\bigl(\norm{f(t)}_{V^*}+M\norm{\Sg u_m(t)}_V\bigr).
\]
Squaring, integrating, and using \eqref{eq:galenergybound} proves the final term in \eqref{eq:galuniform}.
\end{proof}

\begin{theorem}[Galerkin existence]\label{thm:galerkin}
Assume (A1)--(A3).  For every $u_0\in H$ and $f\in L_g^2(I;V^*)$, there exists a unique posterior variational solution $u\in\Wplus$ of \eqref{eq:problem}.  Moreover,
\begin{equation}\label{eq:galweakconv}
 \Sg u_m\rightharpoonup\Sg u\quad\text{in }L_g^2(I;V),
 \qquad
 \Dg u_m\rightharpoonup\Dg u\quad\text{in }L_g^2(I;V^*),
\end{equation}
for the full Galerkin sequence.
\end{theorem}

\begin{proof}
\emph{Step 1: extraction.}
The spaces $L_g^2(I;V)$ and $L_g^2(I;V^*)$ are Hilbert spaces and hence reflexive.  Lemma~\ref{lem:galbounds} therefore gives a subsequence, not relabelled, and functions
\[
 z\in L_g^2(I;V),
 \qquad d\in L_g^2(I;V^*)
\]
such that
\begin{equation}\label{eq:galextract}
 \Sg u_m\rightharpoonup z\quad\text{in }L_g^2(V),
 \qquad
 \Dg u_m\rightharpoonup d\quad\text{in }L_g^2(V^*).
\end{equation}

\emph{Step 2: reconstruction of the candidate state.}
Since $P_mu_0\to u_0$ strongly in $H$ and hence in $V^*$, define
\begin{equation}\label{eq:galcandidate}
 u(t):=u_0+\int_{[a,t)}d(r)\,d\mug(r).
\end{equation}
Then $u\in AC_g^{\rm B}(I;V^*)$ and $\Dg u=d$.  For every $m$,
\[
 \Sg u_m=P_mu_0+\Pgplus(\Dg u_m).
\]
The map $\Pgplus:L_g^2(V^*)\to L_g^2(V^*)$ is bounded linear by Lemma~\ref{lem:Ppm}; hence it is weak-to-weak continuous.  Passing to the limit in the previous identity gives
\[
 \Sg u_m\rightharpoonup u_0+\Pgplus d=\Sg u
 \quad\text{in }L_g^2(V^*).
\]
On the other hand, the first convergence in \eqref{eq:galextract}, composed with the continuous embedding $V\hookrightarrow V^*$, gives $\Sg u_m\rightharpoonup z$ in $L_g^2(V^*)$.  Weak limits are unique, so
\[
 z=\Sg u\quad\text{in }L_g^2(V^*).
\]
Since $z\in L_g^2(V)$, this proves $\Sg u\in L_g^2(V)$ and hence $u\in\Wplus$.

\emph{Step 3: passage to the variational equation.}
Fix $v\in V$ and $\eta\in L_g^\infty(I)$.  In the Galerkin equation use the admissible test vector $P_mv$ and multiply by $\eta(t)$.  Integration gives
\begin{align}\label{eq:galintegrated}
 \int\eta\ip{\Dg u_m}{P_mv}\,d\mug
 +\int\eta a(t;\Sg u_m,P_mv)\,d\mug
 =\int\eta\ip f{P_mv}\,d\mug.
\end{align}
Because $P_mv\to v$ strongly in $V$, the first term converges to $\int\eta\ip d v$.  To see this without suppressing a remainder, write
\begin{align*}
 \int\eta\ip{\Dg u_m}{P_mv}-\int\eta\ip d v
 &=\int\eta\ip{\Dg u_m-d}{v}
   +\int\eta\ip{\Dg u_m}{P_mv-v}.
\end{align*}
The first term tends to zero by weak convergence in $L_g^2(V^*)$; the second is bounded by
\[
 \norm\eta_\infty\norm{\Dg u_m}_{L_g^2(V^*)}
 \norm{P_mv-v}_{L_g^2(V)},
\]
and the last factor is $M_g^{1/2}\norm{P_mv-v}_V\to0$.

For the form term, split similarly:
\begin{align*}
 &\int\eta a(t;\Sg u_m,P_mv)-\int\eta a(t;\Sg u,v)\\
 &=\int\eta a(t;\Sg u_m-\Sg u,v)
   +\int\eta a(t;\Sg u_m,P_mv-v).
\end{align*}
The first term tends to zero because, for fixed $v$ and bounded $\eta$, the function $t\mapsto\eta(t)A(t)^*v$ represents a bounded element of the dual of $L_g^2(V)$; equivalently one can use the bilinear form bound directly.  The second tends to zero by (A2), the uniform $L_g^2(V)$ bound, and $P_mv\to v$ in $V$.  The forcing term is treated in the same way.  Thus \eqref{eq:galintegrated} tends to
\begin{equation}\label{eq:limitintegratedv}
 \int_I\eta(t)\Bigl[\ip{d(t)}v+a(t;\Sg u(t),v)-\ip{f(t)}v\Bigr]d\mug(t)=0.
\end{equation}

Choose a countable dense set $\{v_j\}_{j\ge1}\subset V$.  For each $j$, the scalar bracket in \eqref{eq:limitintegratedv} vanishes almost everywhere because \eqref{eq:limitintegratedv} holds for all bounded scalar $\eta$.  The union of the corresponding null sets is still null.  Outside this common null set the identity holds for every $v_j$.  Since, for fixed $t$, each of the three terms depends continuously on $v\in V$, density extends the identity to all $v\in V$.  Therefore $u$ is a posterior variational solution.

\emph{Step 4: initial value and traces.}
The definition \eqref{eq:galcandidate} gives $u(a)=u_0$ exactly. Proposition~\ref{prop:Pautomatic} supplies the stable chain used in Theorem~\ref{thm:LM}; hence that theorem yields the $H$-valued lateral traces, and Theorem~\ref{thm:energy} yields the exact energy identity.

\emph{Step 5: uniqueness and convergence of the full sequence.}
Corollary~\ref{cor:unique} shows that every subsequential limit obtained from the uniform bounds is the same solution $u$.  To prove convergence of the whole sequence, suppose for contradiction that $\Sg u_m$ does not converge weakly to $\Sg u$ in $L_g^2(V)$.  Then some weak neighborhood $\mathcal U$ of $\Sg u$ excludes a subsequence $(\Sg u_{m_k})$.  The uniform bound permits extraction of a further weakly convergent subsequence.  Repeating Steps 1--4 identifies its limit as the posterior state of a solution of the same problem, hence as $\Sg u$ by uniqueness, contradicting exclusion from $\mathcal U$.  The argument for $\Dg u_m$ is identical, or follows simultaneously from the paired subsequence extraction.  Therefore \eqref{eq:galweakconv} holds for the full sequence.
\end{proof}

The proof architecture has a classical Galerkin analogue in Roub\'i\v{c}ek \cite[Section 8.4, pp.~240--246]{Roubicek2013}, but the present proof contains two genuinely Stieltjes-specific ingredients: the atomic matrix \eqref{eq:galerkinatomicupdate} and the posterior-state identification through the bounded operator $\Pgplus$.

\section{Existence by a Stieltjes-adapted Rothe scheme}\label{sec:rothe}
The Rothe method discretizes the Stieltjes measure rather than the spatial space.  Classical continuous-time Rothe semidiscretization is described by Ka\v{c}ur \cite[pp.~23--34]{Kacur1986} and Roub\'i\v{c}ek \cite[Section 8.2, pp.~215--229]{Roubicek2013}.  Here the partition must simultaneously resolve Euclidean position, Stieltjes mass, and persistent atoms.

\subsection{Partitions adapted to mass and atoms}
The partition must control both Euclidean diameter and Stieltjes mass. A short interval can still contain an atom of fixed positive mass, while a small-mass cell need not be geometrically local. Thus persistent atoms are isolated and both quantities are controlled on residual cells. If there are no atoms, this reduces to an ordinary refining mesh (and to the classical Rothe mesh for $g(t)=t$). With finitely many atoms all may be singleton cells. With countably many atoms, level $n$ isolates only those with mass $>\varepsilon_n$; there are finitely many, and every fixed atom is eventually isolated because $\varepsilon_n\downarrow0$.

\begin{lemma}[Mass--atom partitions]\label{lem:partitions}
There are numbers $\rho_n\downarrow0$, finite ordered families of pairwise disjoint Borel cells $C_{n,1},\ldots,C_{n,K_n}$ of positive mass
\[
 h_{n,k}:=\mug(C_{n,k})>0,
\]
and Borel null sets $Z_n$ such that:
\begin{enumerate}[label=(\roman*)]
\item $[a,b)=Z_n\cup\bigcup_{k=1}^{K_n}C_{n,k}$;
\item every atom with mass greater than $\varepsilon_n\downarrow0$ is a singleton cell;
\item every non-singleton cell satisfies $h_{n,k}\le\rho_n$;
\item the Euclidean diameter of every non-singleton cell tends to zero uniformly in $k$;
\item each fixed atom is a singleton cell for all sufficiently large $n$.
\end{enumerate}
\end{lemma}

\begin{proof}
Fix sequences $\varepsilon_n\downarrow0$ and $\tau_n\downarrow0$.  Let
\[
 A_n:=\{r\in \Ag:\dg(r)>\varepsilon_n\}.
\]
This set is finite. Indeed, if $F\subset A_n$ is any finite subset, then
\[
 \#F\,\varepsilon_n
 <\sum_{r\in F}\dg(r)\le M_g,
\]
so every finite subset has cardinality $<M_g/\varepsilon_n$; hence $A_n$ itself is finite (with the empty case trivial). Declare each $\{r\}$ with $r\in A_n$ to be a singleton cell. These are the large atoms.

Remove their atomic masses and consider the residual finite measure
\[
 \nu_n:=\mug-\sum_{r\in A_n}\dg(r)\delta_r.
\]
Every atom of $\nu_n$ has mass at most $\varepsilon_n$.  Insert as mandatory Euclidean cuts all points of $A_n$, together with a finite ordinary mesh of $[a,b]$ whose mesh size is at most $\tau_n$.  At this stage the singleton sets $\{r\}$, $r\in A_n$, have already been removed from the residual part. Splitting the complement at the same cut points produces finitely many pairwise disjoint semi-open residual intervals $J=[c,d)$ (with the evident endpoint modifications) whose union, together with the declared singletons, covers $[a,b)$ up to no points at all. In particular, no residual cell can overlap a singleton atomic cell, and no large atom lies in the interior of a residual interval.

Fix one such $J$ and write
\[
 F_J(x):=\nu_n([c,x)),\qquad c\le x\le d.
\]
Starting at $c$, introduce a cut whenever the accumulated residual mass first reaches $\varepsilon_n$. If the threshold is crossed at a residual atom $r$, include the singleton $\{r\}$ in the current measurable piece (that is, close the current piece at $r$ in the measure-theoretic partition) and continue with the part to the right of $r$; no real successor of $r$ is assumed. Since every residual atom has mass at most $\varepsilon_n$, the mass collected in any resulting non-singleton piece is at most $2\varepsilon_n$.  The last remainder has mass less than $\varepsilon_n$.  Because $\nu_n(J)<\infty$, only finitely many such cuts occur in $J$.  Repeating this on the finitely many initial intervals produces a finite family.

Any piece of zero $\mug$-mass is removed from the list and absorbed into $Z_n$; the union of finitely many such pieces is still null.  The remaining cells have positive mass and are naturally ordered from left to right.  Set $\rho_n:=2\varepsilon_n$.  Then (iii) holds.  The preliminary ordinary mesh gives (iv).  Finally, if $r$ is a fixed atom, then $\dg(r)>0$, so $\dg(r)>\varepsilon_n$ for all sufficiently large $n$; hence $r\in A_n$ eventually and (v) follows.
\end{proof}

\subsection{The implicit Rothe step}
For every positive-mass cell define the averaged data
\begin{equation}\label{eq:averageddata}
 f_{n,k}:=\frac1{h_{n,k}}\int_{C_{n,k}}f(t)\,d\mug(t),
 \qquad
 a_{n,k}(u,v):=\frac1{h_{n,k}}\int_{C_{n,k}}a(t;u,v)\,d\mug(t).
\end{equation}
The averaged form inherits the same constants:
\[
 |a_{n,k}(u,v)|\le M\norm u_V\norm v_V,
 \qquad
 a_{n,k}(v,v)\ge\alpha\norm v_V^2.
\]
Set $U_{n,0}=u_0$. Here $h_{n,k}=\mug(C_{n,k})$ is the amount of Stieltjes time carried by the cell, and the outgoing state is used implicitly on that cell. Recursively determine $U_{n,k}\in V$ from
\begin{equation}\label{eq:rothestep}
 \left(\frac{U_{n,k}-U_{n,k-1}}{h_{n,k}},v\right)_H
 +a_{n,k}(U_{n,k},v)
 =\ip{f_{n,k}}v,
 \qquad v\in V.
\end{equation}

Equivalently,
\[
 (U_{n,k},v)_H+h_{n,k}a_{n,k}(U_{n,k},v)
 =(U_{n,k-1},v)_H+h_{n,k}\ip{f_{n,k}}v,
\]
so every step is implicit. For $g(t)=t$, $h_{n,k}=t_k-t_{k-1}$ and this is the standard backward-Euler/Rothe step. If $C_{n,k}=\{r\}$ is a singleton atom, then $h_{n,k}=\Delta g(r)$ and the same equation is exactly the posterior variational resolvent. For countably many atoms every fixed atom eventually appears in this exact form by Lemma~\ref{lem:partitions}.

\begin{lemma}[Solvability of every Rothe step]\label{lem:rothestep}
For every $n,k$, equation \eqref{eq:rothestep} has a unique solution $U_{n,k}\in V$.
\end{lemma}

\begin{proof}
Multiplying \eqref{eq:rothestep} by $h_{n,k}$ gives
\[
 (U_{n,k},v)_H+h_{n,k}a_{n,k}(U_{n,k},v)
 =(U_{n,k-1},v)_H+h_{n,k}\ip{f_{n,k}}v.
\]
Define
\[
 b_{n,k}(u,v):=(u,v)_H+h_{n,k}a_{n,k}(u,v).
\]
It is continuous on $V\times V$, and
\[
 b_{n,k}(v,v)
 =\norm v_H^2+h_{n,k}a_{n,k}(v,v)
 \ge h_{n,k}\alpha\norm v_V^2.
\]
The right-hand side is a bounded linear functional on $V$.  Lax--Milgram \cite[Corollary 5.8, p.~140]{Brezis2011} therefore gives one and only one $U_{n,k}\in V$.
\end{proof}

Define
\[
 d_{n,k}:=\frac{U_{n,k}-U_{n,k-1}}{h_{n,k}}.
\]
Let $\bar U_n$ and $d_n$ be the cellwise constant functions equal to $U_{n,k}$ and $d_{n,k}$ on $C_{n,k}$, and set both functions equal to zero on $Z_n$.

\begin{lemma}[Discrete energy and derivative estimates]\label{lem:rothebounds}
There is $C$, independent of $n$, such that
\begin{align}
 \sup_{0\le k\le K_n}\norm{U_{n,k}}_H^2
 &+\sum_{k=1}^{K_n}\norm{U_{n,k}-U_{n,k-1}}_H^2
 +\alpha\norm{\bar U_n}_{L_g^2(V)}^2\notag\\
 &\le 2\norm{u_0}_H^2+2\alpha^{-1}\norm f_{L_g^2(V^*)}^2,\label{eq:rotheE}\\
 \norm{d_n}_{L_g^2(V^*)}&\le C.\label{eq:rotheD}
\end{align}
\end{lemma}

\begin{proof}
Test \eqref{eq:rothestep} with $v=U_{n,k}$ and multiply by $h_{n,k}$.  The Hilbert identity
\[
 (x-y,x)_H=\frac12\bigl(\norm x_H^2-\norm y_H^2+\norm{x-y}_H^2\bigr)
\]
gives
\begin{align*}
 &\frac12\Bigl(\norm{U_{n,k}}_H^2-\norm{U_{n,k-1}}_H^2
 +\norm{U_{n,k}-U_{n,k-1}}_H^2\Bigr)\\
 &\qquad+h_{n,k}a_{n,k}(U_{n,k},U_{n,k})
 =h_{n,k}\ip{f_{n,k}}{U_{n,k}}.
\end{align*}
Coercivity and Young's inequality yield
\[
 h_{n,k}\ip{f_{n,k}}{U_{n,k}}
 \le\frac{h_{n,k}}{2\alpha}\norm{f_{n,k}}_{V^*}^2
 +\frac{\alpha h_{n,k}}2\norm{U_{n,k}}_V^2.
\]
After absorption and summation over $k=1,\ldots,j$, the $H$ terms telescope.  We obtain
\begin{align*}
 \norm{U_{n,j}}_H^2
 &+\sum_{k=1}^j\norm{U_{n,k}-U_{n,k-1}}_H^2
 +\alpha\sum_{k=1}^jh_{n,k}\norm{U_{n,k}}_V^2\\
 &\le\norm{u_0}_H^2
 +\alpha^{-1}\sum_{k=1}^jh_{n,k}\norm{f_{n,k}}_{V^*}^2.
\end{align*}
Jensen's inequality for Bochner averages gives
\[
 h_{n,k}\norm{f_{n,k}}_{V^*}^2
 \le\int_{C_{n,k}}\norm{f(t)}_{V^*}^2\,d\mug(t).
\]
For each $j$, this partial estimate controls $\|U_{n,j}\|_H^2$ by the full-data right-hand side. Taking the supremum in $j$ gives one copy of that bound. Taking $j=K_n$ controls the complete increment sum and coercive term by a second copy. Adding the two estimates proves \eqref{eq:rotheE}.

Let $A_{n,k}:V\to V^*$ be associated with $a_{n,k}$.  Equation \eqref{eq:rothestep} is the $V^*$ identity
\[
 d_{n,k}+A_{n,k}U_{n,k}=f_{n,k}.
\]
Therefore
\[
 \norm{d_{n,k}}_{V^*}
 \le\norm{f_{n,k}}_{V^*}+M\norm{U_{n,k}}_V.
\]
Squaring, multiplying by $h_{n,k}$, summing, and using \eqref{eq:rotheE} and Jensen gives \eqref{eq:rotheD}.
\end{proof}

\subsection{Exact primitive reconstruction and posterior consistency}
Define the exact Stieltjes primitive generated by the discrete derivative,
\begin{equation}\label{eq:unprimitive}
 u_n(t):=u_0+\int_{[a,t)}d_n(s)\,d\mug(s).
\end{equation}
If $t\in C_{n,k}$, let
\[
 m_{n,k}(t):=\mug\bigl(C_{n,k}\cap(-\infty,t]\bigr).
\]
Because the preceding full cells contribute exactly $U_{n,k-1}-u_0$, the posterior primitive satisfies
\begin{equation}\label{eq:posteriorcell}
 \Sg u_n(t)=U_{n,k-1}+m_{n,k}(t)d_{n,k}.
\end{equation}
Since $U_{n,k}=U_{n,k-1}+h_{n,k}d_{n,k}$,
\begin{equation}\label{eq:posteriorerror}
 \Sg u_n(t)-\bar U_n(t)
 =-\bigl(h_{n,k}-m_{n,k}(t)\bigr)d_{n,k}.
\end{equation}
If $C_{n,k}=\{r\}$ is a singleton atomic cell, then $m_{n,k}(r)=h_{n,k}$ and the error is exactly zero.  This exactness on resolved atoms is the reason for isolating them in Lemma~\ref{lem:partitions}.

\begin{lemma}[Posterior consistency]\label{lem:consistency}
\[
 \norm{\Sg u_n-\bar U_n}_{L_g^2(V^*)}\longrightarrow0.
\]
\end{lemma}

\begin{proof}
On a non-singleton cell, $0\le h_{n,k}-m_{n,k}(t)\le h_{n,k}\le\rho_n$.  On a singleton atomic cell the difference is zero.  Hence, using \eqref{eq:posteriorerror},
\begin{align*}
 \norm{\Sg u_n-\bar U_n}_{L_g^2(V^*)}^2
 &\le\rho_n^2\sum_{k=1}^{K_n}h_{n,k}\norm{d_{n,k}}_{V^*}^2\\
 &=\rho_n^2\norm{d_n}_{L_g^2(V^*)}^2.
\end{align*}
The derivative norm is uniformly bounded by \eqref{eq:rotheD} and $\rho_n\to0$.
\end{proof}

\subsection{Cell averages}
For a separable Banach space $Y$ and $\Phi\in L_g^1(I;Y)$ define
\[
 (Q_n\Phi)|_{C_{n,k}}
 :=\frac1{h_{n,k}}\int_{C_{n,k}}\Phi(t)\,d\mug(t),
\]
and set $Q_n\Phi=0$ on $Z_n$.

\begin{lemma}[Strong convergence of cell averages]\label{lem:Qn}
For every separable Banach space $Y$, every $1\le p<\infty$, and every $\Phi\in L_g^p(I;Y)$,
\begin{equation}\label{eq:Qn}
 Q_n\Phi\to\Phi\quad\text{strongly in }L_g^p(I;Y).
\end{equation}
Moreover $\norm{Q_n\Phi}_{L_g^p(Y)}\le\norm\Phi_{L_g^p(Y)}$.
\end{lemma}

\begin{proof}
Jensen's inequality on each cell gives the contraction estimate.  To prove convergence, first let $\Phi$ be continuous on the compact interval with values in $Y$.  On a non-singleton cell $C_{n,k}$, uniform continuity gives
\[
 \sup_{s,t\in C_{n,k}}\norm{\Phi(s)-\Phi(t)}_Y\to0
\]
uniformly in $k$ because the cell diameters tend uniformly to zero.  Hence $Q_n\Phi-\Phi\to0$ uniformly on the union of non-singleton cells.  On singleton cells, $Q_n\Phi=\Phi$ exactly at the atom.  Therefore \eqref{eq:Qn} holds for continuous $\Phi$.

For a finite regular Borel measure on a compact metric space and separable $Y$, continuous $Y$-valued functions are dense in the Bochner space $L^p(\mug;Y)$.  Given arbitrary $\Phi$ and $\varepsilon>0$, choose continuous $\Psi$ with $\norm{\Phi-\Psi}_{L^p}<\varepsilon$.  Then
\[
 \norm{Q_n\Phi-\Phi}_{L^p}
 \le\norm{Q_n(\Phi-\Psi)}_{L^p}
 +\norm{Q_n\Psi-\Psi}_{L^p}
 +\norm{\Psi-\Phi}_{L^p}
 \le2\varepsilon+\norm{Q_n\Psi-\Psi}_{L^p}.
\]
For large $n$ the last term is below $\varepsilon$, proving the claim.
\end{proof}

\begin{theorem}[Rothe existence without (P)]\label{thm:rothe}
Assume (A1)--(A3).  For every $u_0\in H$ and $f\in L_g^2(I;V^*)$, there exists $u\in\Wplus$ solving \eqref{eq:problem}.  Along a subsequence,
\begin{equation}\label{eq:rotheweak}
 \bar U_n\rightharpoonup\Sg u\quad\text{in }L_g^2(I;V),
 \qquad
 d_n\rightharpoonup\Dg u\quad\text{in }L_g^2(I;V^*).
\end{equation}
No spatial projection is used.  By Corollary~\ref{cor:unique}, the solution is unique and therefore equals the Galerkin solution of Theorem~\ref{thm:galerkin}.
\end{theorem}

\begin{proof}
\emph{Step 1: weak extraction.}
The bounds \eqref{eq:rotheE}--\eqref{eq:rotheD} and reflexivity give a subsequence and functions
\[
 z\in L_g^2(I;V),\qquad d\in L_g^2(I;V^*)
\]
such that
\begin{equation}\label{eq:rotheextract}
 \bar U_n\rightharpoonup z\quad\text{in }L_g^2(V),
 \qquad
 d_n\rightharpoonup d\quad\text{in }L_g^2(V^*).
\end{equation}

\emph{Step 2: reconstruct the exact limit state.}
Set
\[
 u(t):=u_0+\int_{[a,t)}d(r)\,d\mug(r).
\]
Then $u\in AC_g^{\rm B}(I;V^*)$, $u(a)=u_0$, and $\Dg u=d$.  Bounded linearity of $\Pgplus$ gives
\[
 \Sg u_n=u_0+\Pgplus d_n
 \rightharpoonup u_0+\Pgplus d=\Sg u
 \quad\text{in }L_g^2(V^*).
\]
By Lemma~\ref{lem:consistency},
\[
 \bar U_n-\Sg u_n\to0
 \quad\text{strongly in }L_g^2(V^*).
\]
The first convergence in \eqref{eq:rotheextract}, followed by the continuous embedding $V\hookrightarrow V^*$, also gives $\bar U_n\rightharpoonup z$ in $L_g^2(V^*)$.  Hence $z=\Sg u$ in $L_g^2(V^*)$.  Since $z\in L_g^2(V)$, the posterior state actually belongs to $L_g^2(V)$ and $u\in\Wplus$.

\emph{Step 3: write the discrete variational identity in integral form.}
Let $\Phi\in L_g^2(I;V)$.  On each cell use the fixed vector $(Q_n\Phi)|_{C_{n,k}}$ as the test vector in \eqref{eq:rothestep}, multiply by $h_{n,k}$, and sum in $k$.  By the definitions of the averages,
\begin{equation}\label{eq:rotheintegrated}
 \int_I\ip{d_n}{Q_n\Phi}\,d\mug
 +\int_I a(t;\bar U_n,Q_n\Phi)\,d\mug
 =\int_I\ip{Q_nf}{Q_n\Phi}\,d\mug.
\end{equation}

\emph{Step 4: pass to the limit term by term.}
By Lemma~\ref{lem:Qn}, $Q_n\Phi\to\Phi$ strongly in $L_g^2(V)$.  Therefore
\begin{align*}
 \int\ip{d_n}{Q_n\Phi}-\int\ip d\Phi
 &=\int\ip{d_n-d}{\Phi}
   +\int\ip{d_n}{Q_n\Phi-\Phi}\longrightarrow0,
\end{align*}
because the first term uses weak convergence of $d_n$ and the second uses the uniform derivative bound and strong convergence of the tests.

For the form term, write
\begin{align*}
 &\int a(t;\bar U_n,Q_n\Phi)-\int a(t;z,\Phi)\\
 &=\int a(t;\bar U_n-z,\Phi)
   +\int a(t;\bar U_n,Q_n\Phi-\Phi).
\end{align*}
The first term tends to zero by weak convergence in $L_g^2(V)$ and the fact that $w\mapsto\int a(t;w(t),\Phi(t))\,d\mug$ is a bounded linear functional, with norm at most $M\norm\Phi_{L_g^2(V)}$.  The second tends to zero because
\[
 \left|\int a(t;\bar U_n,Q_n\Phi-\Phi)\,d\mug\right|
 \le M\norm{\bar U_n}_{L_g^2(V)}\norm{Q_n\Phi-\Phi}_{L_g^2(V)}.
\]
Finally, $Q_nf\to f$ strongly in $L_g^2(V^*)$ by Lemma~\ref{lem:Qn}, so the right-hand side converges to $\int\ip f\Phi$.  We obtain
\begin{equation}\label{eq:rothelimitweak}
 \int_I\ip d\Phi\,d\mug
 +\int_I a(t;z,\Phi)\,d\mug
 =\int_I\ip f\Phi\,d\mug
 \qquad\forall\Phi\in L_g^2(I;V).
\end{equation}

\emph{Step 5: obtain one $V^*$-valued equality, not merely test-dependent scalar equalities.}
By Lemma~\ref{lem:Ameas},
\[
 R:=d+A(\cdot)z-f\in L_g^2(I;V^*).
\]
Equation \eqref{eq:rothelimitweak} says
\[
 \int_I\ip{R(t)}{\Phi(t)}\,d\mug(t)=0
 \qquad\forall\Phi\in L_g^2(I;V).
\]
Recall that $V$ was assumed to be a Hilbert space in the Gelfand triple \eqref{eq:gelfand}. Hence the Riesz map $J_V:V\to V^*$ is an isometric isomorphism. Choose the admissible test function
\[
 \Phi(t):=J_V^{-1}R(t).
\]
Strong measurability follows from bounded linearity of $J_V^{-1}$.  Then
\[
 0=\int_I\ip{R}{J_V^{-1}R}\,d\mug
 =\int_I\norm{J_V^{-1}R(t)}_V^2\,d\mug(t).
\]
Hence $R(t)=0$ in $V^*$ for almost every $t$.  This produces one exceptional null set independent of the test vector.  Recalling $z=\Sg u$ and $d=\Dg u$, we obtain
\[
 \Dg u+A(t)\Sg u=f
 \quad\text{in }V^*,\qquad\mug\text{-a.e.}
\]
Thus $u$ is a posterior variational solution.

No spatial projection occurred in any step.  Corollary~\ref{cor:unique} therefore identifies the Rothe solution with the unique Galerkin solution, while Proposition~\ref{prop:Pautomatic} guarantees the stable chain used by the latter construction.
\end{proof}

\begin{remark}[Logical independence of the two constructions]
The Galerkin and Rothe proofs solve the same linear problem but close their limits in different ways.  Galerkin uses the automatically available stable projections to control the $V^*$ derivative and then identifies the posterior state through $\Pgplus$.  Rothe uses exact primitive reconstruction plus posterior consistency of the cellwise implicit states.  The Rothe proof is therefore not a disguised application of the Galerkin theorem and, in particular, establishes existence without using any spatial projection; Proposition~\ref{prop:Pautomatic} is invoked only when the projection-based energy and Galerkin arguments are used.
\end{remark}

\section{Why anterior Aubin--Lions compactness is not forced here}\label{sec:aubin}
The public version of Fern\'andez \cite{FernandezAubinLions2026} develops the Stieltjes--Bochner fundamental theorem in Theorem 3.4, defines the anterior evolution graph in Definition 4.1, proves the main compactness theorem in Theorem 4.7, gives the Gelfand-triple specialization in Corollary 5.4, and treats a weighted discrete compactness result in Corollary 5.9.  These exact locators refer to the public PDF of arXiv:2609.10064v1.  The state component in Definition 4.1 satisfies an anterior primitive relation with $[a,t)$.

Our natural energy space is different:
\[
 \Sg u\in L_g^2(I;V),\qquad \Dg u\in L_g^2(I;V^*).
\]
It does not, in general, imply $u\in L_g^2(I;V)$.  The following elementary example makes the obstruction explicit.

\begin{proposition}[Posterior $V$-control does not imply anterior $V$-control]\label{prop:postnotant}
Assume that the embedding $V\hookrightarrow H$ is proper, so that $H\setminus V\ne\varnothing$.  Then there are a Stieltjes clock with one atom and a function $u\in\Wplus$ such that
\[
 \Sg u\in L_g^2(I;V),\qquad \Dg u\in L_g^2(I;V^*),
\]
but $u\notin L_g^2(I;V)$.
\end{proposition}

\begin{proof}
Choose $h\in H\setminus V$ and regard $h$ as an element of $V^*$ through the canonical embedding.  Fix $r\in(a,b)$ and $\delta>0$, and choose a left-continuous nondecreasing clock whose measure is the single atom
\[
 \mug=\delta\delta_r.
\]
For instance, after an irrelevant additive constant, take $g(t)=0$ for $t\le r$ and $g(t)=\delta$ for $t>r$.  Define the Bochner density by
\[
 d(r):=-\frac{h}{\delta}\in V^*,
\]
with arbitrary values away from $r$, and set
\[
 u(t):=h+\int_{[a,t)}d(s)\,d\mug(s).
\]
Then $u(a)=h$, $\Dg u=d$, and, because the integral over $[a,r)$ does not include the atom,
\[
 u(r)=h.
\]
The posterior value is
\[
 \Sg u(r)=u(r)+\delta d(r)=h-h=0\in V.
\]
Since the measure is supported only at $r$,
\[
 \norm{\Sg u}_{L_g^2(V)}^2
 =\delta\norm 0_V^2=0,
\]
while
\[
 \norm{\Dg u}_{L_g^2(V^*)}^2
 =\delta\norm{-\frac h\delta}_{V^*}^2
 =\delta^{-1}\norm h_{V^*}^2<\infty.
\]
Thus $u\in\Wplus$.  If $u$ belonged to $L_g^2(I;V)$, then, because the singleton $\{r\}$ has positive measure, its value $u(r)=h$ would have to belong to $V$, contradicting the choice of $h$.  Hence anterior and posterior energy control are genuinely different at atoms.
\end{proof}

Nor can one simply apply the anterior graph theorem to the function $\Sg u$.  If $q=\Dg u$, then
\[
 \Pgplus q(t)-\Pgminus q(t)=\dg(t)q(t),
\]
which can be nonzero only at atoms. Thus the automatically available pair $(\Sg u,q)$ does not satisfy the anterior graph relation of Definition 4.1 in \cite{FernandezAubinLions2026} in general.

This distinction can be seen directly in both approximation procedures.  In Galerkin we know
\[
 \Sg u_m\ \text{bounded in }L_g^2(V),
 \qquad
 \Dg u_m\ \text{bounded in }L_g^2(V^*).
\]
The identity $u_m=\Sg u_m-\dg\Dg u_m$ only controls the second term in $V^*$; it does not produce a uniform $V$-bound for $u_m$.  In Rothe, the cellwise implicit state $\bar U_n$ is uniformly bounded in $L_g^2(V)$, but $(\bar U_n,d_n)$ is not an exact anterior evolution-graph pair.  The exact graph state is $u_n$ from \eqref{eq:unprimitive}, and it is only the posterior state $\Sg u_n$ that approaches $\bar U_n$, in $V^*$, by Lemma~\ref{lem:consistency}.

Therefore Theorem 4.7 or Corollary 5.4 of \cite{FernandezAubinLions2026} should not be inserted into the present proof without an additional hypothesis giving anterior $V$-control.  This causes no difficulty in the linear coercive problem: linearity, weak compactness, and boundedness of $\Pgplus$ suffice to identify the limits.  A posterior-state Aubin--Lions theorem may become important for nonlinear problems, where weak convergence alone may no longer pass through the nonlinear terms.

\section{Comparison with the anterior formulation of 2020}
The anterior theory of \cite{FernandezTojo2020} works mode by mode. For one scalar mode and an atom of mass $\delta$,
\[
 u^+-u=\delta(h-\lambda u),\qquad
 u^+=(1-\lambda\delta)u+\delta h.
\]
Thus $\lambda\delta=1$ does not prevent forward propagation: it resets the outgoing value to $\delta h$. What fails there is invertibility/regressivity of the step and any representation that divides by $1-\lambda\delta$. Theorem~4.2 of arXiv:2001.11183v2 proves the parabolic result under hypotheses (H1)--(H5) \cite[Section~4, Theorems~4.1--4.2]{FernandezTojo2020}.

For the posterior equation the same atomic identity gives
\[
 u^+-u+\delta\lambda u^+=\delta h,
 \qquad
 u^+=\frac{u+\delta h}{1+\lambda\delta}.
\]
Hence the homogeneous factors are $1-\lambda\delta$ and $(1+\lambda\delta)^{-1}$: explicit versus implicit resolvent algebra. Related shifted-state equations on time scales are studied by Braverman and Karpuz \cite{BravermanKarpuz2012}; we use that reference only as contextual prior art, not as support for any proof here.

The distinction is sharper for an unbounded parabolic operator. Let $A_H$ be an autonomous positive selfadjoint unbounded realization on $H$. For any $\delta>0$, $E_\delta=I-\delta A_H:D(A_H)\to H$ cannot extend to a bounded operator on $H$: otherwise $A_H=\delta^{-1}(I-E_\delta)$ would be bounded. In the posterior Gelfand setting, by contrast, the jump is the Lax--Milgram problem associated with $(w,v)_H+\delta a(r;w,v)$, which is coercive for every $\delta>0$. This diagnostic is independent of symmetry. For the posterior problem, the homogeneous resolvent is $H$-contractive by coercivity even when the form is nonsymmetric; a general forcing term is instead measured in $V^*$.

The local Hilbert-space energy algebra points in the same direction. If $d=u^+-u$, then
\[
 2(d,u)_H=\|u^+\|_H^2-\|u\|_H^2-\|d\|_H^2,
 \qquad
 2(d,u^+)_H=\|u^+\|_H^2-\|u\|_H^2+\|d\|_H^2.
\]
Thus incoming-state testing places the quadratic jump term on the unfavorable side, whereas posterior testing yields the positive jump correction in Theorem~\ref{thm:LM}. This is a local algebraic comparison, not a complete anterior energy theorem.

None of this says that every anterior problem with atoms is ill posed. Bounded operators, stronger data, prior continuous smoothing, or specially structured clocks can support restricted anterior theories. Their classification is outside the scope of this paper. Our claims concern the posterior resolvent, the exact posterior Stieltjes--Lions--Magenes identity, and the two posterior existence constructions.

The distinction should also be read against the older operator-valued-measure literature. Ahmed's Section~3 in \cite{Ahmed2007DCDIS} treats linear parabolic systems on a Gelfand triple with an operator-valued measure and Galerkin constructions. Thus we do not claim priority for measure-driven parabolic weak solutions or Galerkin approximation as such.

\section{Model application: the heat equation}\label{sec:heat}
Let $\Omega\subset\R^d$ be a bounded domain.  We work with
\[
 V=H_0^1(\Omega),\qquad H=L^2(\Omega),\qquad V^*=H^{-1}(\Omega),
\]
and the energy norm $\norm v_V:=\norm{\nabla v}_{L^2}$.  Let
\[
 a(u,v):=\int_\Omega\nabla u\cdot\nabla v\,dx.
\]
Then
\[
 |a(u,v)|\le\norm u_V\norm v_V,
 \qquad
 a(v,v)=\norm v_V^2,
\]
so (A1)--(A3) hold with $M=\alpha=1$.  The posterior Stieltjes heat equation is
\begin{equation}\label{eq:heat}
 \Dg u-\Delta\Sg u=f,
 \qquad \Sg u\in L_g^2(I;H_0^1(\Omega)),
 \qquad u(a)=u_0,
\end{equation}
in variational form. Thus the homogeneous Dirichlet condition is imposed on the posterior state; the anterior value at an atom is not required to belong to $H_0^1(\Omega)$.

To verify (P) transparently, assume the standard Dirichlet spectral realization admits an $L^2$-orthonormal eigenbasis $(e_j)$ with $e_j\in H_0^1(\Omega)$ and
\[
 \int_\Omega\nabla e_j\cdot\nabla v\,dx
 =\lambda_j(e_j,v)_{L^2}
 \qquad\forall v\in H_0^1(\Omega),
\]
where $0<\lambda_1\le\lambda_2\le\cdots$ and the span is dense in $H_0^1(\Omega)$ for the energy norm.  These spectral properties are imposed here as part of the model example. For instance, they follow from the standard compact-resolvent realization of the Dirichlet Laplacian whenever the embedding $H_0^1(\Omega)\hookrightarrow L^2(\Omega)$ is compact. Stating the spectral hypothesis explicitly avoids importing a domain-regularity theorem that is irrelevant to the abstract Stieltjes argument.

Set $V_m=\operatorname{span}\{e_1,\ldots,e_m\}$ and let $P_m$ be the $L^2$-orthogonal projection.  If
\[
 v=\sum_{j\ge1}(v,e_j)_{L^2}e_j
\]
in the energy space, then testing the eigenvalue relation shows
\[
 \norm v_V^2
 =\sum_{j\ge1}\lambda_j|(v,e_j)_{L^2}|^2.
\]
Therefore
\begin{align*}
 \norm{P_mv}_V^2
 &=\sum_{j=1}^m\lambda_j|(v,e_j)_{L^2}|^2\\
 &\le\sum_{j\ge1}\lambda_j|(v,e_j)_{L^2}|^2
 =\norm v_V^2.
\end{align*}
Hence (P) holds with $C_P=1$.

At an atom $r$ of mass $\delta$, Proposition~\ref{prop:atom} becomes
\begin{equation}\label{eq:heatatom}
 (u(r^+),v)_{L^2}
 +\delta\int_\Omega\nabla u(r^+)\cdot\nabla v\,dx
 =(u(r),v)_{L^2}+\delta\ip{f(r)}v
\end{equation}
for every $v\in H_0^1(\Omega)$.  Thus each jump is exactly one coercive elliptic resolvent step with step size equal to the atomic mass of the clock.

\begin{corollary}[Posterior Stieltjes heat equation]\label{cor:heat}
Under the spectral assumption just stated, for every $u_0\in L^2(\Omega)$ and $f\in L_g^2(I;H^{-1}(\Omega))$, equation \eqref{eq:heat} has a unique posterior variational solution $u\in\Wplus$.  It satisfies the exact energy identity, including the atomic jump dissipation, and it can be constructed both by the spectral Galerkin method and by the Stieltjes-adapted Rothe scheme.
\end{corollary}

\begin{proof}
The form $a$ satisfies (A1)--(A3), and the preceding calculation proves (P).  The assertion therefore follows from Theorems~\ref{thm:galerkin}, \ref{thm:rothe}, and \ref{thm:energy}.  Formula \eqref{eq:heatatom} is the specialization of the atomic resolvent relation.
\end{proof}

\section{Main theorem and concluding remarks}
\begin{theorem}[Linear posterior Stieltjes evolution]\label{thm:main}
Assume (A1)--(A3). For each $u_0\in H$ and $f\in L_g^2(I;V^*)$, there exists a unique posterior variational solution $u\in\Wplus$ of \eqref{eq:problem}. It satisfies the exact energy identity \eqref{eq:energy}, has $H$-valued lateral traces, is obtained by the Rothe scheme, and is also the weak limit of a Galerkin approximation built on the stable chain supplied by Proposition~\ref{prop:Pautomatic}.
\end{theorem}

\begin{proof}
Existence by Rothe is Theorem~\ref{thm:rothe}. Proposition~\ref{prop:Pautomatic} supplies a stable Galerkin chain in every Gelfand triple considered here; Theorem~\ref{thm:LM} then gives the $H$ traces and the exact energy identity, Corollary~\ref{cor:unique} gives uniqueness, and Theorem~\ref{thm:galerkin} gives the Galerkin construction.
\end{proof}

One question is left deliberately open: whether there is a natural posterior-state Aubin--Lions graph whose state component is $\Sg u$ rather than $u$. Such a theorem is likely to matter for nonlinear extensions but is unnecessary for the linear coercive problem treated here. The stable chain needed by the Galerkin and energy arguments is automatically available by Proposition~\ref{prop:Pautomatic}; no claim is made that an arbitrary preassigned Galerkin chain must be $V$-stable.

\end{document}